\documentclass[12pt,reqno]{amsart}
\usepackage[margin=1in]{geometry}

\newcommand{\DateOfPub}[1]{}

\usepackage[UKenglish]{babel}

\usepackage{amsfonts,amsmath,amsthm} % \text and \mathbb command
\usepackage{mathrsfs} % script type font
\usepackage{stmaryrd} % new bracket
\usepackage{mathtools}

\usepackage{lscape} %to make the page landscape

\usepackage{natbib}
\usepackage{url}

\usepackage{tikz-cd}
\usepackage{ascmac}
\usepackage{soul}

\usepackage{comment} % commenting out multiple lines
\usepackage{graphicx} % graphics
\usepackage{color}

\usepackage{multirow}

\allowdisplaybreaks[1]
\newcommand\EatDot[1]{}

\newtheorem{theorem}{Theorem}[subsection]
\newtheorem{proposition}[theorem]{Proposition}
\newtheorem{lemma}[theorem]{Lemma}
\newtheorem{corollary}[theorem]{Corollary}
\theoremstyle{remark}
\newtheorem{remark}{Remark}
\theoremstyle{definition}

\newcommand{\E}{\mathbf{E}}
\newcommand{\N}{\mathbf{N}}
\newcommand{\Q}{\mathbf{Q}}
\newcommand{\R}{\mathbf{R}}

\newcommand{\diff}{\mathrm{d}}
\newcommand{\sgn}{\mathrm{sgn}}
\newcommand{\loc}{\mathrm{loc}}

\newcommand{\Wasserstein}{\mathcal{W}}
\newcommand{\Beta}{\mathrm{B}}
\newcommand{\FMC}{\mathbb{M}}
\newcommand{\zero}{\mathbf{0}}
\newcommand{\indicator}{\mathbb{I}}
\newcommand{\Borel}{\mathcal{B}}

\newcommand{\bias}{\mathbf{b}}
\newcommand{\variance}{\mathbf{v}}

\newcommand{\ucpi}{c_{\mathrm{P}}(\pi)}

\title[Langevin-type Monte Carlo]{Non-asymptotic analysis of Langevin-type Monte Carlo algorithms}
\author{Shogo Nakakita}
\address{Komaba Institute for Science, University of Tokyo, 3-8-1 Komaba, Meguro, Tokyo 153-8902, Japan}

\begin{document}
\maketitle

\begin{abstract}
    We study Langevin-type algorithms for sampling from Gibbs distributions such that the potentials are dissipative and their weak gradients have finite moduli of continuity not necessarily convergent to zero.
    Our main result is a non-asymptotic upper bound of the 2-Wasserstein distance between a Gibbs distribution and the law of general Langevin-type algorithms based on the Liptser--Shiryaev theory and Poincar\'{e} inequalities.
    We apply this bound to show that the Langevin Monte Carlo algorithm can approximate Gibbs distributions with arbitrary accuracy if the potentials are dissipative and their gradients are uniformly continuous.
    We also propose Langevin-type algorithms with spherical smoothing for distributions whose potentials are not convex or continuously differentiable.
\end{abstract}

\section{Introduction}
We consider the problem of sampling from a Gibbs distribution $\pi(\diff x)\propto \exp(-\beta U(x))\diff x$ on $(\R^{d},\Borel(\R^{d}))$, where $U:\R^{d}\to[0,\infty)$ is a non-negative potential function and $\beta>0$ is the inverse temperature.
One of the extensively used types of algorithms for the sampling is the Langevin type motivated by the Langevin dynamics, the solution of the following $d$-dimensional stochastic differential equation (SDE):
\begin{align}\label{eq:LD:Intro}
    \diff X_{t}=-\nabla U\left(X_{t}\right)\diff t+\sqrt{2\beta^{-1}}\diff B_{t},\ X_{0}=\xi,
\end{align}
where $\{B_{t}\}_{t\ge0}$ is a $d$-dimensional Brownian motion and $\xi$ is a $d$-dimensional random vector with $|\xi|<\infty$ almost surely.
Since the 2-Wasserstein or total variation distance between $\pi$ and the law of $X_{t}$ is convergent under mild conditions, we expect that the laws of Langevin-type algorithms inspired by $X_{t}$ should converge to $\pi$.
However, most of the theoretical guarantees for the algorithms are based on the convexity of $U$, the twice continuous differentiability of $U$, or the Lipschitz continuity of the gradient $\nabla U$, which do not hold in some modelling in statistics and machine learning.
The main interest of this study is a unified approach to analyse and propose Langevin-type algorithms under minimal assumptions.

The stochastic gradient Langevin Monte Carlo (SG-LMC) algorithm or stochastic gradient Langevin dynamics with a constant step size $\eta>0$ is the discrete observations $\{Y_{i\eta}\}_{i=0,\ldots,k}$ of the solution of the following $d$-dimensional SDE:
\begin{align}\label{eq:SGLMC}
    \diff Y_{t}=-{G}\left(Y_{\lfloor t/\eta\rfloor\eta},a_{\lfloor t/\eta\rfloor\eta}\right)\diff t+\sqrt{2\beta^{-1}}\diff B_{t},\ Y_{0}=\xi,
\end{align}
where $\{a_{i\eta}\}_{i=0,\ldots,k}$ is a sequence of independent and identically distributed (i.i.d.) random variables on a measurable space $(A,\mathcal{A})$ and ${G}$ is a $\R^{d}$-valued measurable function.
We assume that $\{a_{i\eta}\}$, $\{B_{t}\}$, and $\xi$ are independent.
Note that the Langevin Monte Carlo (LMC) algorithm is a special case of SG-LMC; it has the representation as the discrete observations $\{Y_{i\eta}\}_{i=0,\ldots,k}$ of the solution of the following diffusion-type SDE:
\begin{align}\label{eq:LMC}
    \diff Y_{t}=-\nabla U\left(Y_{\lfloor t/\eta\rfloor\eta}\right)\diff t+\sqrt{2\beta^{-1}}\diff B_{t},\ Y_{0}=\xi,
\end{align}
which corresponds to the case ${G}=\nabla U$.

To see what difficulties we need to deal with, we review a typical analysis \citep{raginsky2017non} based on the smoothness of $U$, that is, the twice continuous differentiability of $U$ and the Lipschitz continuity of $\nabla U$.
Firstly, the twice continuous differentiability simplifies discussions or plays significant roles in studies of functional inequalities such as Poincar\'{e} inequalities and logarithmic Sobolev inequalities \citep[e.g.,][]{bakry2008simple,cattiaux2010note}.
Since the functional inequalities for $\pi$ are essential in analysis of Langevin algorithms, the assumption that $U$ is of class $\mathcal{C}^{2}$ frequently appears in previous studies.
In the second place, the Lipschitz continuity combined with weak conditions ensures the representation of the likelihood ratio between $\{X_{t}\}$ and $\{Y_{t}\}$, which is critical when we bound the Kullback--Leibler divergence.
\citet{liptser2001statistics} exhibit much weaker conditions than Novikov's or Kazamaki's condition for the explicit representation if \eqref{eq:LD:Intro} has the unique strong solution.
Since the Lipschitz continuity of $\nabla U$ is sufficient for the existence and the uniqueness of the strong solution of \eqref{eq:LD:Intro}, the framework of \citet{liptser2001statistics} is applicable.

Our approaches to overcome the non-smoothness of $U$ are mollification, a classical approach to dealing with non-smoothness in differential equations \citep[e.g., see][]{menozzi2021density,menozzi2022heat}, and the abuse of moduli of continuity for possibly discontinuous functions.
We consider the convolution $\Bar{U}_{r}:=U\ast \rho_{r}$ on $U$ with a weak gradient, and some sufficiently smooth non-negative function $\rho_{r}$ with compact support in a ball of centre $\zero$ and radius $r\in(0,1]$.
We can let $\Bar{U}_{r}$ be of class $\mathcal{C}^{2}$ and obtain bounds for the constant of Poincar\'{e} inequalities for $\Bar{\pi}^{r}(\diff x)\propto \exp(-\beta\Bar{U}_{r}(x))\diff x$, which suffice to show the convergence of the law of the mollified dynamics $\{\Bar{X}_{t}^{r}\}$ defined by the SDE
\begin{align*}
    \diff \Bar{X}_{t}^{r}=-\nabla \Bar{U}_{r}\left(\Bar{X}_{t}^{r}\right)\diff t +\sqrt{2\beta^{-1}}\diff B_{t},\ \Bar{X}_{0}^{r}=\xi
\end{align*} 
to the corresponding Gibbs distribution $\Bar{\pi}^{r}$ in 2-Wasserstein distance owing to \citet{bakry2008simple}, \citet{liu2020poincare}, and \citet{lehec2023langevin}.
Since the convolution $\nabla \Bar{U}_{r}$ is Lipschitz continuous if the modulus of continuity of a representative $\nabla U$ is finite (the convergence to zero is unnecessary), a concise representation of the likelihood ratios between the mollified dynamics $\{\Bar{X}_{t}^{r}\}$ and $\{Y_{t}\}$ is available, and we can evaluate the Kullback--Leibler divergence under weak assumptions.

As our analysis relies on mollification, the bias--variance decomposition of $G$ with respect to $\nabla \Bar{U}_{r}$ rather than $\nabla U$ is crucial.
This decomposition gives us a unified approach to analyse well-known Langevin-type algorithms and propose new algorithms for $U$ without continuous differentiability.
Concretely speaking, we show that the sampling error of LMC under the dissipativity of $U$ of class $\mathcal{C}^{1}$ and uniformly continuous $\nabla U$ can be arbitrarily small by controlling $k$, $\eta$, and $r$ carefully and letting the bias converge.
We also propose new algorithms named the spherically smoothed Langevin Monte Carlo (SS-LMC) algorithm and the spherically smoothed stochastic gradient Langevin Monte Carlo (SS-SG-LMC) algorithm, whose errors can be arbitrarily small under the dissipativity of $U$ and the boundedness of the modulus of continuity of weak gradients.
In addition, we argue zeroth-order versions of these algorithms which are naturally obtained via integration by parts.

\subsection{Related works}

Non-asymptotic analysis of Langevin-based algorithms under convex potentials has been one of the subjects of much attention and intense research \citep{Dal17,durmus2017nonasymptotic,durmus2019high}, and one without convexity has also gathered keen interest \citep{raginsky2017non,XCZ+18,erdogdu2018global}.
Whilst most previous studies are based on the Lipschitz continuity of the gradients of potentials, several studies extend the settings to those without global Lipschitz continuity.
We can classify the settings of potentials in those studies into three types: (1) potentials with convexity but without smoothness \citep{pereyra2016proximal,chatterji2020langevin,lehec2023langevin}; (2) potentials with H\"{o}lder continuous gradients and degenerate convexity at infinity or outside a ball \citep{erdogdu2021convergence,nguyen2022unadjusted,chewi2022analysis}; and (3) potentials with local Lipschitz gradients \citep{brosse2019tamed,zhang2023nonasymptotic}.
We review the results (1) and (2) as our study gives the error estimate of LMC with uniformly continuous gradients and Langevin-type algorithms with gradients whose discontinuity is uniformly bounded.

\citet{pereyra2016proximal}, \citet{chatterji2020langevin}, and \citet{lehec2023langevin} study Langevin-type algorithms under the convexity and the non-smoothness of potentials.
\citet{pereyra2016proximal} presents proximal Langevin-type algorithms for potentials with convexity but without smoothness, which use the Moreau approximations and proximity mappings instead of the gradients.
The algorithms are stable in the sense that they have exponential ergodicity for arbitrary step sizes.
\citet{chatterji2020langevin} propose the perturbed Langevin Monte Carlo algorithm for non-smooth potential functions and show its performance to approximate Gibbs distributions.
The difference between perturbed LMC and ordinary LMC is the inputs of the gradients; we need to add Gaussian noises not only to the gradients but also to their inputs.
The main idea of the algorithm is to use Gaussian smoothing of potential functions studied in \citet{nesterov2017random}; the expectation of non-smooth convex potentials with inputs perturbed by Gaussian random vectors is smoother than the potentials themselves.
\citet{lehec2023langevin} investigates the projected LMC for potentials with convexity, global Lipschitz continuity and discontinuous bounded gradients.
The analysis is based on convexity and estimate for local times of diffusion processes with reflecting boundaries.
The study also generalizes the result to potentials with local Lipschitz by considering a ball as the support of the algorithm and letting the radius diverge.

\citet{erdogdu2021convergence}, \citet{chewi2022analysis}, and \citet{nguyen2022unadjusted} estimate the error of LMC under non-convex potentials with degenerate convexity, weak smoothness, and weak-dissipativity.
\citet{erdogdu2021convergence} show convergence guarantees of LMC under the degenerate convexity at infinity and weak dissipativity of potentials with H\"{o}lder continuous gradients, which are the assumptions for modified logarithmic Sobolev inequalities.
\citet{nguyen2022unadjusted} relaxes the condition of \citet{erdogdu2021convergence} by considering the degenerate convexity outside a large ball and the mixture weak smoothness of potential functions.
\citet{chewi2022analysis} analyse the convergence with respect to the R\'{e}nyi divergence under either Lata\l{}a--Oleszkiewics inequalities or modified logarithmic Sobolev inequalities.

Note that our proof of the results uses approaches similar to the smoothing of \citet{chatterji2020langevin} and the control of the radius of \citet{lehec2023langevin}, whilst our motivations and settings are close to those of the studies under non-convexity.

\subsection{Contributions}
Theorem \ref{thm:sglmc}, the main theoretical result of this paper, gives an upper bound for the 2-Wasserstein distance between the law of general SG-LMC given by Eq.~\eqref{eq:SGLMC} and the target distribution $\pi$ under weak conditions.
We assume the weak differentiability of $U$ combined with the boundedness of the modulus of continuity of a weak gradient $\nabla U$ rather than the twice continuous differentiability of $U$ and the Lipschitz continuity of $\nabla U$.
The generality of the assumptions results in a concise and general framework for analysis of Langevin-type algorithms.
We demonstrate the strength of this framework through analysis of LMC under weak smoothness and proposal for new Langevin-type algorithms without the continuous differentiability or convexity of $U$.

Our contribution to analysis of LMC is to show a theoretical guarantee of LMC under non-convexity and weak smoothness in a direction different to the previous studies.
The main difference between our assumptions and those of the other studies under non-convex potentials is whether to assume (a) the strong dissipativity of the potentials and the uniform continuity of the gradients or (b) the degenerate convexity of the potentials and the H\"{o}lder continuity of the gradients.
Our analysis needs neither the degenerate convexity nor the H\"{o}lder continuity, whilst we need the dissipativity stronger than those assumed in the previous studies.
Since the assumptions (a) and (b) do not imply each other, our main contribution on analysis of LMC is not to strengthen the previous studies but to broaden the theoretical guarantees of LMC under weak smoothness in a different direction.

Moreover, our proposal for Langevin-type algorithms with non-asymptotic error estimates for potentials without convexity, continuous differentiability, or bounded gradients is also a significant contribution.
The proposed algorithms are useful for sampling from posterior distributions for some modelling in statistics and machine learning whose potentials are dissipative and weakly differentiable but neither convex nor continuously differentiable (e.g., some losses with elastic net regularization in nonlinear regression and robust regression).
Furthermore, we can use the zeroth-order versions of them inspired by the recent study of \citet{roy2022stochastic} for black-box sampling with guaranteed accuracy from distributions whose potentials are not convex or smooth.

\subsection{Outline}
We introduce the outline of this paper.
Section \ref{sec:main} displays the main theorem and its concise representation.
In Section \ref{sec:algorithms}, we apply the result to analysis of LMC and proposal for Langevin-type algorithms.
Section \ref{sec:preliminary} is devoted to preliminary results.
We finally present the proof of the main theorem in Section \ref{sec:proof}.

\subsection{Notations}
Let $|\cdot|$ denote the Euclidean norm of $\R^{\ell}$ for all $\ell\in\N$.
$\langle \cdot,\cdot\rangle $ is the Eucllidean inner product of $\R^{\ell}$.
$\|\cdot\|_{2}$ is the spectral norm of matrices, which equals the largest singular value.
For arbitrary matrix $A$, $A^{\top}$ denotes the transpose of $A$.
For all $x\in\R^{\ell}$ and $R>0$, let $B_{R}(x)$ and $\Bar{B}_{R}(x)$ be an open ball and a closed one of centre $x$ and radius $R$ with respect to the Euclidean metric respectively.
We use the notation $\|f\|_{\infty}:=\sup_{x\in\R^{d}}|f(x)|$ for arbitrary $f:\R^{d}\to\R^{\ell}$ and  $d,\ell\in\R^{d}$.

For arbitrary two probability measures $\mu$ and $\nu$ on $(\R^{d},\Borel(\R^{d}))$ and $p\ge1$, we define the $p$-Wasserstein distance between $\mu$ and $\nu$ such that
\begin{align*}
    \Wasserstein_{p}\left(\mu,\nu\right):=\left(\inf_{\pi\in\Pi(\mu,\nu)}\int_{\R^{d}\times\R^{d}}\left|x-y\right|^{p}\diff \pi\left(x,y\right)\right)^{\frac{1}{p}},
\end{align*}
where $\Pi(\mu,\nu)$ is the set of all couplings for $\mu$ and $\nu$.
We also define $D\left(\mu\|\nu\right)$ and $\chi^{2}\left(\mu\|\nu\right)$, the Kullback--Leibler divergence and the $\chi^{2}$-divergence of $\mu$ from $\nu$ with $\mu \ll \nu$ respectively such that
\begin{align*}
    D\left(\mu\|\nu\right)=\int_{\R^{d}} \log\left(\frac{\diff \mu}{\diff \nu}\right)\diff \mu,\ \chi^{2}\left(\mu\|\nu\right)=\int_{\R^{d}} \left(\frac{\diff \mu}{\diff \nu}-1\right)^{2}\diff \nu.
\end{align*}

\section{Main results}\label{sec:main}
This section gives the main theorem for non-asymptotic estimates of the error of general SG-LMC algorithms in 2-Wasserstein distance.

\subsection{Estimate of the errors of general SG-LMC}
We consider a compact polynomial mollifier \citep{anderson2014compact} $\rho:\R^{d}\to[0,\infty)$ of class $\mathcal{C}^{1}$ as follows:
\begin{align}\label{eq:mollifier}
    \rho(x)=\begin{cases}
        \left(\frac{\pi^{d/2}\Beta(d/2,3)}{\Gamma(d/2)}\right)^{-1}\left(1-|x|^{2}\right)^{2}&\text{if }|x|\le 1,\\
        0 & \text{otherwise,}
    \end{cases}
\end{align}
where $\Beta(\cdot,\cdot)$ is the beta function and $\Gamma(\cdot)$ is the gamma function.
Note that $\nabla \rho$ has an explicit $L^{1}$-bound, which is the reason to adopt $\rho$ as the mollifier in our analysis; we give more detailed discussions on $\rho$ in Section \ref{sec:mollifier}.
Let $\rho_{r}(x)=r^{-d}\rho(x/r)$ with $r>0$.

We define $\widetilde{G}(x)$ such that for each $x\in\R^{d}$,
\begin{align*}
    \widetilde{G}(x):=\E\left[{G}\left(x,a_{0}\right)\right],
\end{align*}
whose measurability is given by Tonelli's theorem.

We set the following assumptions on $U$ and $G$.
\begin{enumerate}
    \item[(A1)] $U\in W_{\loc}^{1,\infty}(\R^{d})$.
    \item[(A2)] For each $a\in A$ and $x\in\R^{d}$, $|G(x,a)|<\infty$.
\end{enumerate}
Under (A1), we fix a representative $\nabla U$ and consider the assumptions on $\nabla U$ and $\widetilde{G}$.
\begin{enumerate}
    \item[(A3)] $|\nabla U(\zero)|<\infty$ and $|\widetilde{G}(\zero)|<\infty$, and the moduli of continuity of $\nabla U$ and $\widetilde{G}$ are bounded, that is,
    \begin{align*}
        \omega_{\nabla U}(r)&:=\sup_{x,y\in\R^{d}:|x-y|\le r}\left|\nabla U(x)-\nabla U(y)\right|<\infty,\\
        \omega_{\widetilde{G}}(r)&:=\sup_{x,y\in\R^{d}:|x-y|\le r}\left|\widetilde{G}(x)-\widetilde{G}(y)\right|<\infty
    \end{align*}
    for some $r\in(0,1]$.
    \item[(A4)] For some $m,\Tilde{m},b,\Tilde{b}>0$, for all $x\in\R^{d}$,
    \begin{align*}
        \left\langle x,\nabla U\left(x\right)\right\rangle \ge m\left|x\right|^{2}-b,\
        \left\langle x,\widetilde{G}\left(x\right)\right\rangle \ge \Tilde{m}\left|x\right|^{2}-\Tilde{b}.
    \end{align*}
\end{enumerate}

\begin{remark}
    The boundedness of the moduli of continuity in Assumption (A3) is equivalent to the boundedness for all $r>0$ or for some $r>0$; see Lemma \ref{lem:FMC:representation}.
    Note that we allow $\lim_{r\downarrow0}\omega_{\nabla U}(r)\neq0$ and $\lim_{r\downarrow0}\omega_{\widetilde{G}}(r)\neq0$.
\end{remark}
Under (A1) and (A3), we can define the mollification 
\begin{align*}
    \nabla \Bar{U}_{r}(x):=\nabla \left(U\ast \rho_{r}\right)(x)=\nabla \int_{\R^{d}} U\left(y\right)\rho_{r}\left(x-y\right)\diff y=\left(\nabla U\ast \rho_{r}\right)(x),
\end{align*}
where the last equality holds since (A1) gives the essential boundedness of $U$ and $\nabla U$ on any compact sets and we can approximate $\rho_{r}\in C_{0}^{1}(\R^{d})\cap W^{1,1}(\R^{d})$ with some $\{\varphi_{n}\}\subset\mathcal{C}_{0}^{\infty}(\R^{d})$.
Note that $\Bar{U}_{r}\in \mathcal{C}^{2}(\R^{d})$ with $\nabla^{2}\Bar{U}_{r}=\nabla U\ast \nabla\rho_{r}$ by this discussion (see Lemma \ref{lem:FMC:C2}).
We assume the quadratic growths of the bias of $G$ with respect to $\nabla \Bar{U}_{r}$ and the variance as well.
\begin{enumerate}
    \item[(A5)] For some $\Bar{\delta}>0$ and $\delta_{r}:=(\delta_{\bias,r,0},\delta_{\bias,r,2},\delta_{\variance,0},\delta_{\variance,2})\in[0,\Bar{\delta}]^{4}$, for almost all $x\in\R^{d}$, 
    \begin{align*}
        \left|\widetilde{G}(x)-\nabla \Bar{U}_{r}(x)\right|^{2}&\le 2\delta_{\bias,r,2}\left|x\right|^{2}+2\delta_{\bias,r,0},\\
        \E\left[\left|{G}\left(x,a_{0}\right)-\widetilde{G}(x)\right|^{2}\right]&\le 2\delta_{\variance,2}\left|x\right|^{2}+2\delta_{\variance,0}.
    \end{align*}
\end{enumerate}
For brevity, we use the notation $\delta_{r,i}=\delta_{\bias,r,i}+\delta_{\variance,i}$ for both $i=0,2$.
We also give the condition on the initial value $\xi$.
\begin{enumerate}
    \item[(A6)] The initial value $\xi$ has the law $\mu_{0}(\diff x)=(\int_{\R^{d}} \exp(-\Psi(x))\diff x)^{-1}\exp(-\Psi(x))\diff x$ with $\Psi:\R^{d}\to[0,\infty)$ and $\psi_{0},\psi_{2}>0$ such that $(2\vee \beta(|\nabla U(\zero)|+\omega_{\nabla U}(1)))|x|^{2}-\psi_{0}\le \Psi(x)\le \psi_{2}|x|^{2}+\psi_{0}$ for all $x\in\R^{d}$.
\end{enumerate}
Assumption (A6) yields
\begin{align*}
    \kappa_{0}:=\log\int_{\R^{d}}e^{|x|^{2}}\mu_{0}(\diff x)<\infty.
\end{align*}

Let $\mu_{t}$ with $t\ge0$ denote the probability measure of $Y_{t}$.
The following theorem gives an upper bound for the 2-Wasserstein distance between $\mu_{k\eta}$ and $\pi$; its proof is given in Section \ref{sec:proof}.

\begin{theorem}[error estimate of general SG-LMC]\label{thm:sglmc}
    Assume (A1)--(A6) and $\eta\in(0,1\wedge(\Tilde{m}/2((\omega_{\widetilde{G}}(1))^{2}+\delta_{\variance,2}))]$.
    It holds that for any $r\in(0,1]$ and $k\in\N$,
    \begin{align*}
        \Wasserstein_{2}\left(\mu_{k\eta},\pi\right)\le &2C_{1}\left.\left(x^{\frac{1}{2}}+ x^{\frac{1}{4}}\right)\right|_{x=f(\delta_{r},r,k,\eta)}+\sqrt{2C_{2}c_{\rm P}(\Bar{\pi}_{r})}e^{-k\eta/2\beta c_{\rm P}\left(\Bar{\pi}_{r}\right)},
    \end{align*}
    where $f$ is the function defined as
    \begin{align*}
        f(\delta_{r},r,k,\eta):=\left(C_{0}\frac{\omega_{\nabla U}(r)}{r}\eta+\beta\left(\delta_{r,2}\kappa_{\infty}+\delta_{r,0}\right)\right)k\eta+\beta r\omega_{\nabla U}(r),
    \end{align*}
    $C_{0},C_{1},C_{2},\kappa_{\infty}>0$ are the positive constants defined as
    \begin{align*}
        C_{0}&:=\left(d+2\right)\left(\frac{\beta}{3}\left(\left\|\widetilde{G}\right\|_{\FMC}^{2}+\delta_{\variance,0}+\left((\omega_{\widetilde{G}}(1))^{2}+\delta_{\variance,2}\right)\kappa_{\infty}\right)+\frac{d}{2}\right),\\
        C_{1}&:=2\sqrt{4\kappa_{0}+\frac{16(\Bar{b}+d/\beta)}{m}+\frac{10}{1\wedge (\beta m/4)}},\\
        C_{2}&:=3^{\beta b/2}\left(\frac{3\psi_{2}}{m\beta}\right)^{d/2}\exp\left(\beta\left(2\left\|\nabla U\right\|_{\FMC}+U_{0}\right)+2\psi_{0}\right),\\
        \kappa_{\infty}&:=\kappa_{0}+2\left(1\vee \frac{1}{\Tilde{m}}\right)\left(\Tilde{b}+\left\|\widetilde{G}\right\|_{\FMC}^{2}+\delta_{\variance,0}+\frac{d}{\beta}\right),
    \end{align*}
    $\Bar{b}:=b+\omega_{\nabla U}(1)$, $U_{0}:=\left\|U\right\|_{L^{\infty}(B_{1}(\zero))}$, $\|\nabla U\|_{\FMC},\|\widetilde{G}\|_{\FMC}>0$ are the positive constants defined as
    \begin{align*}
        \left\|\nabla U\right\|_{\FMC}:=\left|\nabla U(\zero)\right|+\omega_{\nabla U}(1),\     \left\|\widetilde{G}\right\|_{\FMC}:=\left|\widetilde{G}(\zero)\right|+\omega_{\widetilde{G}}(1),
    \end{align*}
    $c_{\rm P}(\Bar{\pi}^{r})>0$ is the constant of a Poincar\'{e} inequality for $\Bar{\pi}^{r}(\diff x)\propto \exp(-\beta\Bar{U}_{r}(x))\diff x$ such that
    \begin{align*}
        c_{\rm P}(\Bar{\pi}^{r})
        &\le \frac{2}{m\beta\left(d+\Bar{b}\beta\right)}+\frac{4a\left(d+\Bar{b}\beta\right)}{m\beta}\exp\left(\beta\left(\frac{3}{2}\left\|\nabla U\right\|_{\FMC}\left(1+\frac{4\left(d+\Bar{b}\beta\right)}{m\beta}\right)+U_{0}\right)\right),
    \end{align*}
    and $a>0$ is a positive absolute constant.
\end{theorem}

\subsection{Concise representation of Theorem \ref{thm:sglmc}}
Since the constants and the upper bounds for some of them in Theorem \ref{thm:sglmc} depend on various parameters, we give a concise representation of the result for the error analyses below.
Assume that $f(\delta_{r},r,k,\eta)\le 1$ and $\eta\in(0,1\wedge (\Tilde{m}/2((\omega_{\widetilde{G}}(1))^{2}+\delta_{\variance,2}))]$ and note that Lemma \ref{lem:FMC:esssupOsc} and the perturbation theory \citep{bakry2014analysis} yield that $\exp(-2\omega_{\nabla U}(1))c_{\rm P}(\Bar{\pi}^{r})\le \ucpi\le \exp(2\omega_{\nabla U}(1))c_{\rm P}(\Bar{\pi}^{r})$ for any $r\in(0,1]$, 
where $\ucpi$ is the Poincare constant of $\pi$.
We then obtain that for some $C\ge 1$ independent of $\delta_{r},r,k,\eta,d,\ucpi$,
\begin{align*}%\label{eq:asymp:sglmc}
        \Wasserstein_{2}\left(\mu_{k\eta},\pi\right)&\le C\sqrt{d}\sqrt[4]{\left(d^{2}\frac{\omega_{\nabla U}(r)}{r}\eta+d\delta_{r,2}+\delta_{r,0}\right)k\eta+r\omega_{\nabla U}(r)}\\
        &\quad+e^{Cd}\exp\left(-\frac{k\eta}{C\ucpi}\right).
\end{align*}
\begin{remark}
    Whilst $\ucpi=\mathcal{O}(\exp(\mathcal{O}(d)))$ in general, there are some known structures to relax the dependence on dimension.
    (i) If $U$ is $\lambda$-strongly-convex with $\lambda>0$, then $\ucpi\le 1/(\beta\lambda)$.
    (ii) \citep[perturbation theory; see][]{bakry2014analysis} If $U=F+V$ with essentially bounded $F$ and $\lambda$-strongly-convex $V$  with $\lambda>0$, then $\ucpi\le \exp(2\beta \|F\|_{L^{\infty}})/(\beta\lambda)=\mathcal{O}(1)$.
    (iii) \citep[Miclo's trick; see][]{bardet2018functional} If $U=U_{l}+U_{c}$ with $M$-Lipschitz continuous $U_{l}$ with $M>0$ and $\lambda$-strongly-convex $V\in\mathcal{C}^{2}$ with $\lambda>0$, then $\ucpi\le (4/(\beta\lambda))\exp(4\beta M^{2}\sqrt{2d}/(\lambda\sqrt{\pi}))=\mathcal{O}(\exp(\mathcal{O}(\sqrt{d})))$.
\end{remark}

\section{Sampling complexities of Langevin-type algorithms}\label{sec:algorithms}
We analyse LMC and the algorithms named spherically smoothed LMC (SS-LMC) and spherically smoothed SG-LMC (SS-SG-LMC) to show their sampling complexities for achieving $\Wasserstein_{2}(\mu_{k\eta},\pi)\le \epsilon$ with arbitrary $\epsilon>0$.
We also discuss zeroth-order versions of SS-LMC and SS-SG-LMC.

\subsection{Analysis of the LMC algorithm for $U$ of class $\mathcal{C}^{1}$ with the uniformly continuous gradient}
We examine the LMC algorithm for $U$ with the uniformly continuous gradient, that is, $\omega_{\nabla U}(r)\to 0$ as $r\to0$.
Under the LMC algorithm, we use $G=\nabla U$ and thus $\widetilde{G}=\nabla U$.
Therefore, the bias--variance decomposition in (A5) is given as $\delta_{\bias,r,0}=(\omega_{\nabla U}(r))^{2}/2$, $\delta_{\bias,r,2}=\delta_{\variance,0}=\delta_{\variance,2}=0$ by Lemma \ref{lem:FMC:lipMap} below.

We present the assumptions in this section.
\begin{enumerate}
    \item[(B1)] $U\in\mathcal{C}^{1}(\R^{d})$.
    \item[(B2)] $\nabla U$ is uniformly continuous, that is, the modulus of continuity $\omega_{\nabla U}$ defined as
    \begin{align*}
        \omega_{\nabla U}\left(r\right):=\sup_{x,y\in\R^{d}:|x-y|\le r}\left|\nabla U(x)-\nabla U(y)\right|<\infty
    \end{align*}
    with $r\ge0$ is continuous at zero.
    \item[(B3)] There exist $m,b>0$ such that for all $x\in\R^{d}$,
    \begin{align*}
        \left\langle x,\nabla U(x)\right\rangle \ge m\left|x\right|^{2}-b.
    \end{align*}
\end{enumerate}

(A1)--(A4) hold immediately by (B1)--(B3); therefore, we yield the following corollary.
\begin{corollary}[error estimate of LMC]
Under (B1)--(B3) and (A6), there exists a constant $C\ge1$ independent of $r,\alpha,k,\eta,d,\ucpi$ such that for all $k\in\N$, $\eta\in(0,1\wedge (m/2(\omega_{\nabla U}(1))^{2})]$, and $r\in(0,1]$ with $\left(d^{2}(\omega_{\nabla U}(r)/r)\eta+(\omega_{\nabla U}(r))^{2}\right)k\eta+r\omega_{\nabla U}(r)\le 1$, 
\begin{align*}%\label{eq:asymp:lmc}
    \Wasserstein_{2}\left(\mu_{k\eta},\pi\right)&\le C\sqrt{d}\sqrt[4]{\left(d^{2}\frac{\omega_{\nabla U}(r)}{r}\eta+\left(\omega_{\nabla U}\left(r\right)\right)^{2}\right)k\eta+r\omega_{\nabla U}(r)}\\
    &\quad+e^{Cd}\exp\left(-\frac{k\eta}{C\ucpi}\right).
\end{align*}
\end{corollary}

\subsubsection{The sampling complexity of the LMC algorithm}
We present the propositions regarding the sampling complexity to achieve the approximation $\Wasserstein_{2}(\mu_{k\eta},\pi)\le \epsilon$ for arbitrary $\epsilon>0$.
Define generalized inverses of $\omega_{\nabla U}$ as follows: for any $s>0$,
\begin{align*}
    \omega_{\nabla U}^{\dagger}\left(s\right):=\sup\left\{r\ge 0:\omega_{\nabla U}(r)\le s\right\}.
\end{align*}
The continuity of $\nabla U$ under (B2) along with its monotonicity gives $\omega_{\nabla U}(\omega_{\nabla U}^{\dagger}(s))= s$.
We also define a generalized inverse of $r\omega_{\nabla U}(r)$ such that for all $s>0$,
\begin{align*}
    \iota\left(s\right):=\sup\left\{r\ge 0:r\omega_{\nabla U}(r)\le s\right\}.
\end{align*}

The following proposition yields the sampling complexity using this generalized inverse.
\begin{proposition}\label{prop:LMC:1}
    Assume that (B1)--(B3) and (A6) hold and fix $\epsilon\in(0,1]$.
    We set $\Bar{r}_{1},\Bar{r}_{2}>0$ such that
    \begin{align*}
        \Bar{r}_{1}:=\omega_{\nabla U}^{\dagger}\left(\sqrt{\frac{\epsilon^{4}}{48C^{4}d^{2}\left(C\ucpi\left(\log\left(2/\epsilon\right)+Cd\right)+1\right)}}\right),\ \Bar{r}_{2}:=\iota\left(\frac{\epsilon^{4}}{48C^{4}d^{2}}\right).
    \end{align*}
    If $r= \Bar{r}_{1}\wedge \Bar{r}_{2}$ and 
    \begin{align*}
        \eta\le 1&\wedge \frac{m}{2\left(\omega_{\nabla U}(1)\right)^{2}}\wedge \left(\frac{r}{\omega_{\nabla U}(r)}\frac{\epsilon^{4}}{48C^{4}d^{4}\left(C\ucpi\left(\log\left(2/\epsilon\right)+Cd\right)+1\right)}\right),
    \end{align*}
    then $\Wasserstein_{2}\left(\mu_{k\eta},\pi\right)\le \epsilon$ with $k=\lceil C\ucpi\left(\log(2/\epsilon)+Cd\right)/\eta\rceil$.
\end{proposition}

\begin{proof}
We just need to confirm
\begin{align*}
    \max\left\{d^{2}\frac{\omega_{\nabla U}\left(r\right)}{r}k\eta^{2},\left(\omega_{\nabla U}\left(r\right)\right)^{2}k\eta,r\omega_{\nabla U}\left(r\right)\right\}\le \frac{\epsilon^{4}}{48C^{4}d^{2}},\ e^{Cd}\exp\left(-\frac{k\eta}{C\ucpi}\right)\le \frac{\epsilon}{2}.
\end{align*}
$r\omega_{\nabla U}\left(r\right)\le \epsilon^{4}/48C^{4}d^{2}$ is immediate.
Since $\eta\le 1$, we have
\begin{align*}
    C\ucpi\left(\log\left(2/\epsilon\right)+Cd\right)\le k\eta\le C\ucpi\left(\log\left(2/\epsilon\right)+Cd\right)+1,
\end{align*}
and the other bounds also hold.
\end{proof}

We can apply Proposition \ref{prop:LMC:1} to analysis of the sampling complexity of LMC with $\alpha$-mixture weakly smooth gradients \citep{chatterji2020langevin,nguyen2022unadjusted}.
Assume that there exist $M>0$ and $\alpha\in(0,1]$ such that for all $x,y\in\R^{d}$,
\begin{align*}
    \left|\nabla U(x)-\nabla U(y)\right|\le M\left(\left|x-y\right|^{\alpha}\vee\left|x-y\right|\right),
\end{align*}
which is a weaker assumption than both $\alpha$-H\"{o}lder continuity and Lipschitz continuity.
This allows the gradient $\nabla U(x)$ to be at most of linear growth, whilst $\alpha$-H\"{o}lder continuity with $\alpha\in(0,1)$ lets the gradient be at most of sublinear growth.
Since $\omega_{\nabla U}(r)\le M(r^{\alpha}\vee r)$ for all $r\ge0$,
it holds $\omega_{\nabla U}^{\dagger}(s)\ge (s/M)^{1/\alpha}$ for $s\in(0,1/M]$.
Rough estimates of $r/\omega_{\nabla U}(r)$ by the inequalities $r/\omega_{\nabla U}(r)\ge r^{1-\alpha}/M$, $\omega_{\nabla U}^{\dagger}(s)/\omega_{\nabla U}(\omega_{\nabla U}^{\dagger}(s))= \omega_{\nabla U}^{\dagger}(s)/s\ge s^{1/\alpha-1}/M^{1/\alpha}$, $\iota(s)\ge(s/M)^{1/(1+\alpha)}$, and $\iota(s)/\omega_{\nabla U}(\iota(s))\ge \iota(s)^{1-\alpha}/M\ge s^{(1-\alpha)/(1+\alpha)}/M^{2/(1+\alpha)}$ for sufficiently small $r,s>0$ yield the sampling complexity
\begin{align*}
    k=\mathcal{O}\left(\frac{d^{4}\ucpi^{2}\left(\log\epsilon^{-1}+d\right)^{2}}{\epsilon^{4}}\left(\left(\frac{d^{2}\ucpi\left(\log\epsilon^{-1}+d\right)}{\epsilon^{4}}\right)^{\frac{1-\alpha}{2\alpha}}\vee \left(\frac{d^{2}}{\epsilon^{4}}\right)^{\frac{1-\alpha}{1+\alpha}}\right)\right).
\end{align*}

\subsection{The spherically smoothed Langevin Monte Carlo (SS-LMC) algorithm}\label{sec:SSLMC}
We consider a stochastic gradient $G$ unbiased for $\nabla \Bar{U}_{r}$ with fixed $r\in(0,1]$ such that the sampling error can be sufficiently small.

Note that $\rho$ is the density of random variables which we can generate as a product of a random variable following the uniform distribution on $S^{d-1}=\{x\in\R^{d}:|x|=1\}$ and the square root of a random variable following the beta distribution $\textrm{Beta}(d/2,3)$ independently.
Therefore, we can consider spherical smoothing with the random variables whose density is $\rho_{r}$ as an analogue to Gaussian smoothing of \citet{chatterji2020langevin}.

Set the stochastic gradient 
\begin{align*}
    {G}\left(x,a_{i\eta}\right)=\frac{1}{N_{B}}\sum_{j=1}^{N_{B}}\nabla U\left(x+r'\zeta_{i,j}\right),
\end{align*}
where $N_{B}\in\N$, $r'\in(0,1]$, $a_{i\eta}=[\zeta_{i,1},\ldots,\zeta_{i,N_{B}}]$ and $\{\zeta_{i,j}\}$ is a sequence of i.i.d.~random variables with the density $\rho$.
Then for any $x\in\R^{d}$, $\E[{G}(x,a_{i\eta})]=\nabla \Bar{U}_{r'}(x)$,
\begin{align*}
    &\E\left[\left|{G}\left(x,a_{i\eta}\right)-\nabla \Bar{U}_{r'}(x)\right|^{2}\right]\\
    &=\frac{1}{N_{B}}\int _{\R^{d}}\left|\nabla U(x-y)-\nabla \Bar{U}_{r'}(x)\right|^{2}\rho_{r'}(y)\diff y\\
    &\le \frac{1}{N_{B}}\int_{\R^{d}}\int_{\R^{d}} \left|\nabla U(x-y)-\nabla U(x-z)\right|^{2}\rho_{r'}(y)\rho_{r'}(z)\diff y\diff z\\
    &\le \frac{1}{N_{B}}\int_{\R^{d}}\int_{\R^{d}} \left(\left|\nabla U(x-y)-\nabla U(x)\right|+\left|\nabla U(x-z)-\nabla U(x)\right|\right)^{2}\rho_{r'}(y)\rho_{r'}(z)\diff y\diff z\\
    &\le \frac{\left(2\omega_{\nabla U}(r')\right)^{2}}{N_{B}}
\end{align*}
by Jensen's inequality, and (A5) holds if $\nabla \Bar{U}_{r'}(x)$ is well-defined and $\omega_{\nabla U}(r')<\infty$ exists.

The main idea is to let $r'=r$, where $r$ is the radius of the implicit mollification and $r'$ is the radius of the support of the random noises which we control.
Hence, the stochastic gradient $G$ with $r'=r$ is an unbiased estimator of the mollified gradient $\nabla \Bar{U}_{r}(x)$.
We call the algorithm with this $G$ the spherically smoothed Langevin Monte Carlo (SS-LMC) algorithm.
We can control the sampling error of SS-LMC to be close to zero by taking a sufficiently small $r$.

Let us set the following assumptions.
\begin{enumerate}
    \item[(C1)] $U\in W_{\loc}^{1,\infty}(\R^{d})$.
    \item[(C2)] $|\nabla U(\zero)|<\infty$ and the modulus of continuity of $\nabla U$ is bounded, that is,
    \begin{align*}
        \omega_{\nabla U}(r):=\sup_{x,y\in\R^{d}:|x-y|\le r}\left|\nabla U(x)-\nabla U(y)\right|<\infty
    \end{align*}
    for some $r\in(0,1]$.
    \item[(C3)] There exist $m,b>0$ such that for all $x\in\R^{d}$,
    \begin{align*}
        \left\langle x,\nabla U\left(x\right)\right\rangle &\ge m\left|x\right|^{2}-b.
    \end{align*}
\end{enumerate}

Let us observe that (C1)--(C3) yield (A1)--(A5).
(A1) is the same as (C1).
(C2) yields (A2) by Lemma \ref{lem:FMC:linearGrowth} and (A3) by $|\nabla \Bar{U}_{r}(\zero)|\le |\nabla U(\zero)|+\omega_{\nabla U}(1)$ and $\omega_{\nabla \Bar{U}_{r}}(1)\le 3\omega_{\nabla U}(1)<\infty$.
(A4) also holds since 
\begin{align*}
    \left\langle x,\nabla \Bar{U}_{r}\left(x\right)\right\rangle&\ge m|x|^{2}-(b+\omega_{\nabla U}(1));
\end{align*}
Section \ref{sec:proof} gives the detailed derivation of this inequality.
(A5) is given by (C2) and the discussion above.

\begin{corollary}[error estimate of SS-LMC]
    Under (C1)--(C3) and (A6), there exists a constant $C\ge 1$ independent of $N_{B},r,k,\eta,d,\ucpi$ such that for all $k\in\N$, $\eta\in(0,1\wedge (m/(4(\omega_{\nabla U}(1))^{2}))]$, $r\in(0,1]$, and $N_{B}\in\N$ with $(d^{2}(\omega_{\nabla U}(r)/r)\eta+(\omega_{\nabla U}(r))^{2}/N_{B})k\eta+r\omega_{\nabla U}(r)\le 1$, 
    \begin{align*}
        \Wasserstein_{2}\left(\mu_{k\eta},\pi\right)&\le C\sqrt{d}\sqrt[4]{\left(d^{2}\frac{\omega_{\nabla U}(r)}{r}\eta+\frac{(\omega_{\nabla U}(r))^{2}}{N_{B}}\right)k\eta+r\omega_{\nabla U}(r)}\\
        &\quad+e^{Cd}\exp\left(-\frac{k\eta}{C\ucpi}\right).
    \end{align*}
\end{corollary}

\subsubsection{The sampling complexity of SS-LMC}
We analyse the behaviour of SS-LMC; to see that the convergence $\omega_{\nabla U}(r)\downarrow 0$ is unnecessary, we consider a rough version of the upper bound by replacing $\omega_{\nabla U}(r)$ with the constant $\omega_{\nabla U}(1)$.
\begin{corollary}
Under (C1)--(C3) and (A6), there exists a constant $C\ge1$ independent of $N_{B}$, $r$, $k$, $\eta$, $d$, and $\ucpi$ such that for all $N_{B}\in\N$, $k\in\N$, $\eta\in(0,1\wedge (m/(4(\omega_{\nabla U}(1))^{2}))]$, and $r\in(0,1]$ with $\left(d^{2}r^{-1}\eta+N_{B}^{-1}\right)k\eta+r\le 1$,
\begin{align*}%\label{eq:asymp:sslmc}
    \Wasserstein_{2}\left(\mu_{k\eta},\pi\right)&\le C\sqrt{d}\sqrt[4]{\left(d^{2}r^{-1}\eta+N_{B}^{-1}\right)k\eta+r}+e^{Cd}\exp\left(-\frac{k\eta}{C\ucpi}\right).
\end{align*}
\end{corollary}

We yield the following estimate of the sampling complexity: the proof is identical to Proposition \ref{prop:LMC:1}.
\begin{proposition}\label{prop:SSLMC:comp}
    Assume (C1)--(C3) and (A6) and fix $\epsilon\in(0,1]$.
    If $r=\epsilon^{4}/48C^{4}d^{2}$, $N_{B}\ge 48C^{4}d^{2}(C\ucpi(\log(2/\epsilon)+Cd)+1)/\epsilon^{4}$, and $\eta$ satisfies
    \begin{align*}
        \eta\le 1&\wedge \frac{m}{4\left(\omega_{\nabla U}(1)\right)^{2}}\wedge \frac{r\epsilon^{4}}{48C^{4}d^{4}(C\ucpi(\log(2/\epsilon)+Cd)+1)},
    \end{align*}
    then $\Wasserstein_{2}\left(\mu_{k\eta},\pi\right)\le \epsilon$ for $k=\lceil C\ucpi(\log(2/\epsilon)+Cd)/\eta\rceil$.
\end{proposition}

Since the complexities of $N_{B}$ and $k$ are given as $N_{B}=\mathcal{O}(d^{2}\ucpi(\log\epsilon^{-1}+d)/\epsilon^{4})$ and $k=\mathcal{O}(d^{6}\ucpi^{2}(\log\epsilon^{-1}+d)^{2}/\epsilon^{8})$, 
% \begin{align*}
%     N_{B}=\mathcal{O}\left(\frac{d^{2}\ucpi\left(\log\epsilon^{-1}+d\right)}{\epsilon^{4}}\right),\
%     k=\mathcal{O}\left(\frac{d^{13/2}\ucpi^{2}\left(\log\epsilon^{-1}+d\right)^{2}}{\epsilon^{8}}\right),
% \end{align*}
we obtain the sampling complexity of SS-LMC as $N_{B}k=\mathcal{O}(d^{8}\ucpi^{3}(\log\epsilon^{-1}+d)^{3}/\epsilon^{12})$ or $N_{B}k=\widetilde{\mathcal{O}}(d^{11}\ucpi^{3}/\epsilon^{12})$, where $\widetilde{\mathcal{O}}$ ignores logarithmic factors.
% \begin{align*}
%     N_{B}k=\mathcal{O}\left(\frac{d^{17/2}\ucpi^{3}\left(\log\epsilon^{-1}+d\right)^{3}}{\epsilon^{12}}\right).
% \end{align*}

\subsubsection{Examples of distributions with the regularity conditions}
We show a simple class of potential functions satisfying (C1)--(C3) and some examples in Bayesian inference; assume $\beta=1$ for simplicity of interpretation.
Let us consider a possibly non-convex loss with elastic net regularization such that
\begin{align*}
    U\left(x\right)=L\left(x\right)+\frac{\lambda_{1}}{\sqrt{d}}R_{1}\left(x\right)+\lambda_{2}R_{2}\left(x\right),
\end{align*}
where $L:\R^{d}\to[0,\infty)$ is in $W_{\loc}^{1,\infty}(\R^{d})$ with a weak gradient $\nabla L$ satisfying $\|\nabla L\|_{\infty}<\infty$, $\lambda_{1}\ge 0$, $\lambda_{2}>0$, $R_{1}(x)=\sum_{i=1}^{d}|x^{(i)}|$ with $x^{(i)}$ indicating the $i$-th component of $x$, and $R_{2}(x)=|x|^{2}$.
Fix a weak gradient of $R_{1}$ as $\nabla R_{1}(x)=(\sgn(x^{(1)}),\ldots,\sgn(x^{(d)}))$; then $\omega_{\nabla U}(1)\le 2(\|\nabla L\|_{\infty}+\lambda_{1}+\lambda_{2})<\infty$ and $\langle x,\nabla U(x)\rangle\ge \lambda_{2}|x|^{2}-\|\nabla L\|_{\infty}^{2}/4\lambda_{2}$ since $\langle x,\nabla R_{1}(x)\rangle \ge 0$ for all $x\in\R^{d}$.
Note that regularization corresponds to the potentials of prior distributions in Bayesian inference; for instance, letting $\lambda_{1}=0$ is equivalent to choosing a Gaussian prior $N(\zero,(2\lambda_{2})^{-1}I_{d})$ on $x$.

Non-convex losses with bounded weak gradients often appear in nonlinear and robust regression.
We first examine a squared loss for nonlinear regression (or equivalently nonlinear regression with Gaussian errors) such that $L_{\rm NLR}(x)=\sum_{\ell=1}^{N}\left(y_{\ell}-\phi_{\ell}\left(x\right)\right)^{2}/2\sigma^{2}$, where $N\in\N$, $\sigma>0$ is fixed, $y_{\ell}\in\R$, and $\phi_{\ell}\in W_{\loc}^{1,\infty}(\R^{d})$ with $\|\phi_{\ell}\|_{\infty}+\|\nabla \phi_{\ell}\|_{\infty}<\infty$ for some $\nabla \phi_{\ell}$
(e.g., a two-layer neural network with clipped ReLU activation such that $\phi_{\ell}(x)=(1/W)\sum_{w=1}^{W}a_{w}\varphi_{[0,c]}(\langle x_{w},f_{\ell}\rangle)$, where $\varphi_{[0,c]}(t)=(0\vee t)\wedge c$ with $t\in\R$, $a_{w}\in\{-1,1\}$ and $c>0$ are fixed, $f_{\ell}\in\R^{F}$, $x=(x_{1},\ldots,x_{W})\in \R^{FW}$, $F,W\in\N$, and $d=FW$).
This $L_{\rm NLR}$ indeed satisfies $\|\nabla L_{\rm NLR}\|_{\infty}\le \sum_{\ell=1}^{N}(|y_{\ell}|+\|\phi_{\ell}\|_{\infty})\|\nabla \phi_{\ell}\|_{\infty}/\sigma^{2}<\infty$.
Another example is a Cauchy loss for robust linear regression (or equivalently linear regression with Cauchy errors) such that $L_{\rm RLR}(x)=\sum_{\ell=1}^{N}\log(1+|y_{\ell}-\langle f_{\ell},x\rangle|^{2}/\sigma^{2})$, where $N\in\N$, $\sigma>0$ is fixed, $y_{\ell}\in\R$, and $f_{\ell}\in\R^{d}$.
The fact $|\frac{\diff}{\diff t}\log(1+t^{2}/\sigma^{2})|=|2t/(t^{2}+\sigma^{2})|\le 1/\sigma$ for all $t\in\R$ yields $\|\nabla L_{\rm RLR}\|_{\infty}\le \sum_{\ell=1}^{N}|f_{\ell}|/\sigma<\infty$.

\subsection{The spherically smoothed stochastic gradient Langevin Monte Carlo (SS-SG-LMC) algorithm}
We consider a sampling algorithm for potentials such that for some $N\in\N$,
\begin{align}
    U\left(x\right)=\frac{1}{N}\sum_{\ell=1}^{N}U_{\ell}\left(x\right),\label{eq:emploss}
\end{align}
where $U_{\ell}(x)$ are non-negative functions with the following assumptions.
\begin{enumerate}
    \item[(D1)] For all $\ell=1,\ldots,N$, $U_{\ell}\in W_{\loc}^{1,\infty}(\R^{d})$.
    \item[(D2)] $|\nabla U_{\ell}(\zero)|<\infty$ for all $\ell=1,\ldots,N$ and there exists a function $\Hat{\omega}:[0,\infty)\to[0,\infty)$ such that for all $r\in(0,1]$ and $\ell=1,\ldots,N$,
    \begin{align*}
        \sup_{x,y\in\R^{d}:|x-y|\le r}\left|\nabla U_{\ell}(x)-\nabla U_{\ell}(y)\right|\le \Hat{\omega}\left(r\right)<\infty.
    \end{align*}
    \item[(D3)] There exist $m,b>0$ such that for all $x\in\R^{d}$,
    \begin{align*}
        \left\langle x,\nabla U\left(x\right)\right\rangle &\ge m\left|x\right|^{2}-b.
    \end{align*}
\end{enumerate}

We define the stochastic gradient
\begin{align*}
    {G}\left(x,a_{i\eta}\right)=\frac{1}{N_{B}}\sum_{j=1}^{N_{B}}\nabla U_{\lambda_{i,j}}\left(x+r'\zeta_{i,j}\right),
\end{align*}
where $N_{B}\in\N$, $r'\in(0,1]$, $a_{i\eta}=[\lambda_{i,1},\ldots,\lambda_{i,N_{B}},\zeta_{i,1},\ldots,\zeta_{i,N_{B}}]$, $\{\lambda_{i,j}\}$ is a sequence of i.i.d.~random variables with the discrete uniform distribution on the integers $1,\ldots,N$, and $\{\zeta_{i,j}\}$ is a sequence of i.i.d.~random variables with the density $\rho$ and independence of $\{\lambda_{i,j}\}$.
It holds that for any $x\in\R^{d}$, $\E[{G}(x,a_{i\eta})]=\nabla \Bar{U}_{r'}(x)$ and
\begin{align*}
    \E\left[\left|{G}\left(x,a_{i\eta}\right)-\nabla \Bar{U}_{r'}(x)\right|^{2}\right]&= \frac{1}{N_{B}}\E\left[\left|\nabla U_{\lambda_{i,j}}\left(x+r'\zeta_{i,j}\right)-\nabla \Bar{U}_{r'}(x)\right|^{2}\right]\\
    &\le \frac{1}{N_{B}}\E\left[\left|\nabla U_{\lambda_{i,j}}\left(x+r'\zeta_{i,j}\right)\right|^{2}\right]\\
    &\le \frac{2}{N_{B}}\max_{\ell=1,\ldots,N}((\omega_{\nabla U_{\ell}}(1))^{2}|x|^{2}+(|\nabla U_{\ell}(\zero)|+2\omega_{\nabla U_{\ell}}(1))^{2})
\end{align*}
by Lemma \ref{lem:FMC:linearGrowth}.
We obtain (A5) with $\delta_{\bias,r',0}=\delta_{\bias,r',2}=0$, $\delta_{\variance,0}=(\max_{\ell}|\nabla U_{\ell}(\zero)|+2\Hat{\omega}(1))^{2})/N_{B}$, and $\delta_{\variance,2}=(\Hat{\omega}(1))^{2}/N_{B}$.
We call this algorithm the spherically smoothed stochastic gradient Langevin Monte Carlo (SS-SG-LMC) algorithm.

(D1)--(D3) yield (A1)--(A5) with the same discussion as for SS-LMC.

\begin{corollary}[error estimate of SS-SG-LMC]
    Under (D1)--(D3) and (A6), there exists a constant $C\ge 1$ independent of $N_{B},r,k,\eta,d,\ucpi$ such that for all $k\in\N$, $\eta\in(0,1\wedge (m/(8(\Hat{\omega}(1))^{2}))]$, $r\in(0,1]$, and $N_{B}\in\N$ with $(d^{2}(\Hat{\omega}(r)/r)\eta+d/N_{B})k\eta+r\Hat{\omega}(r)\le 1$, 
    \begin{align*}
        \Wasserstein_{2}\left(\mu_{k\eta},\pi\right)\le C\sqrt{d}\sqrt[4]{\left(d^{2}\frac{\Hat{\omega}(r)}{r}\eta+\frac{d}{N_{B}}\right)k\eta+r\Hat{\omega}(r)}+e^{Cd}\exp\left(-\frac{k\eta}{C\ucpi}\right).
    \end{align*}
\end{corollary}

\subsubsection{The sampling complexity of SS-SG-LMC}
We study the sampling complexity of SS-SG-LMC; we give a rough upper bound by replacing $\Hat{\omega}(r)$ with the constant $\Hat{\omega}(1)$ as the discussion on SS-LMC.
\begin{corollary}
Under (D1)--(D3) and (A6), there exists a constant $C\ge1$ independent of $N_{B}$, $r$, $k$, $\eta$, $d$, and $\ucpi$ such that for all $N_{B}\in\N$, $k\in\N$, $\eta\in(0,1\wedge (m/(8(\Hat{\omega}(1))^{2}))]$, and $r\in(0,1]$ with $\left(d^{2}r^{-1}\eta+dN_{B}^{-1}\right)k\eta+r\le 1$,
\begin{align*}%\label{eq:asymp:sssglmc}
    \Wasserstein_{2}\left(\mu_{k\eta},\pi\right)&\le C\sqrt{d}\sqrt[4]{\left(d^{2}r^{-1}\eta+dN_{B}^{-1}\right)k\eta+r}+e^{Cd}\exp\left(-\frac{k\eta}{C\ucpi}\right).
\end{align*}
\end{corollary}

We yield the following estimate of the sampling complexity, which is lower than that of SS-LMC for $U$ given by Eq.~\eqref{eq:emploss} if $N>d$ since the complexity to compute $G$ in SS-LMC for this $U$ increases by a factor of $N$ and the sampling complexity of SS-SG-LMC deteriorates by a factor of $d$ in comparison to that of SS-LMC.

\begin{proposition}
    Assume (D1)--(D3) and (A6) and fix $\epsilon\in(0,1]$.
    If $r=\epsilon^{4}/48C^{4}d^{2}$, $N_{B}\ge 48C^{4}d^{3}(C\ucpi(\log(2/\epsilon)+Cd)+1)/\epsilon^{4}$, and $\eta$ satisfies
    \begin{align*}
        \eta\le 1&\wedge \frac{m}{8\left(\Hat{\omega}(1)\right)^{2}}\wedge \frac{r\epsilon^{4}}{48C^{4}d^{4}(C\ucpi(\log(2/\epsilon)+Cd)+1)},
    \end{align*}
    then $\Wasserstein_{2}\left(\mu_{k\eta},\pi\right)\le \epsilon$ for $k=\lceil C\ucpi(\log(2/\epsilon)+Cd)/\eta\rceil$.
\end{proposition}

\subsection{Zeroth-order Langevin algorithms}

Let us consider a zeroth-order version of SS-LMC as an analogue to \citet{roy2022stochastic} with the following $G$ under (C1)--(C3) and the assumption $|U(x)|<\infty$ for all $x\in\R^{d}$:
\begin{align*}
    G(x,a_{i\eta})=\frac{1}{N_{B}}\sum_{j=1}^{N_{B}}G_{j}(x,a_{i\eta}):=\frac{1}{N_{B}}\sum_{j=1}^{N_{B}}\frac{U\left(x+r'\zeta_{i,j}\right)-U\left(x\right)}{r'}\frac{4\zeta_{i,j}}{\left(1-|\zeta_{i,j}|^{2}\right)},
\end{align*}
where $N_{B}\in\N$, $r'\in(0,1]$, and $\{\zeta_{i,j}\}$ is an i.i.d.~sequence of random variables with the density $\rho$.
The fact that
\begin{align*}
    \frac{U\left(x+r'\zeta_{i,j}\right)-U\left(x\right)}{r'}\frac{4\zeta_{i,j}}{\left(1-|\zeta_{i,j}|^{2}\right)}=\frac{U\left(x+r'\zeta_{i,j}\right)-U\left(x\right)}{r'}\frac{-\nabla \rho\left(\zeta_{i,j}\right)}{\rho\left(\zeta_{i,j}\right)},
\end{align*}
the symmetricity of $\rho$, approximation of $\rho\in \mathcal{C}_{0}^{1}(\R^{d})\cap W^{1,1}(\R^{d})$, and the essential boundedness of $U$ and $\nabla U$ on compact sets by Lemmas \ref{lem:FMC:linearGrowth} and \ref{lem:FMC:quadGrowth} yield that for all $x\in\R^{d}$,
\begin{align*}
    \E\left[G_{j}(x,a_{i\eta})\right]%&=\E\left[\frac{U\left(x+r'\zeta_{i,j}\right)-U\left(x\right)}{r'}\frac{-\nabla \rho\left(\zeta_{i,j}\right)}{\rho\left(\zeta_{i,j}\right)}\right]\\
    &=\int_{\R^{d}}\frac{U\left(x+r'z\right)-U\left(x\right)}{r'}\frac{-\nabla \rho\left(z\right)}{\rho\left(z\right)}\rho\left(z\right)\diff z\\
    &=-\int_{\R^{d}}\frac{U\left(x+r'z\right)-U\left(x\right)}{r'}\nabla \rho\left(z\right)\diff z\\
    &=-\int_{\R^{d}}\left(U\left(x+y\right)-U\left(x\right)\right)\left(\frac{1}{(r')^{d+1}}\nabla \rho\left(\frac{y}{r'}\right)\right)\diff y\\
    &=\int_{\R^{d}}\nabla U\left(x+y\right)\rho_{r'}\left(y\right)\diff y\\
    %&=\int_{\R^{d}}\nabla U\left(x-y\right)\rho_{r'}\left(y\right)\diff y\\
    &=\nabla \Bar{U}_{r'}\left(x\right).
\end{align*}
Lemma \ref{lem:FMC:localLipschitz}, the convexity of $f(a)=a^{2}$ with $a\in\R$, and the equality
\begin{align*}
    \int_{B_{1}(\zero)}\frac{\left|\nabla \rho\left(z\right)\right|^{2}}{\rho\left(z\right)}\diff z%&=\frac{16\Gamma(d/2)}{\pi^{d/2}\Beta(d/2,3)}\int_{B_{1}(\zero)}\frac{\left|z\right|^{2}\left(1-\left|z\right|^{2}\right)^{2}}{\left(1-\left|z\right|^{2}\right)^{2}}\diff z\\
    &=\frac{16\Gamma(d/2)}{\pi^{d/2}\Beta(d/2,3)}\int_{B_{1}(\zero)}\left|z\right|^{2}\diff z
    =\frac{32}{\Beta(d/2,3)}\int_{0}^{1}r^{d+1}\diff r\\
    &=\frac{32\Gamma(d/2+3)}{\Gamma(d/2)\Gamma(3)(d+2)}
    =\frac{16(d/2+2)(d/2+1)(d/2)}{(d+2)}=2d(d+4)
\end{align*}
give that for almost all $x\in\R^{d}$,
\begin{align*}
    \E\left[\left|G_{1}(x,a_{i\eta})\right|^{2}\right]%&=\E\left[\left|\frac{U\left(x+r'\zeta_{i,j}\right)-U\left(x\right)}{r'}\frac{-\nabla \rho\left(\zeta_{i,j}\right)}{\rho\left(\zeta_{i,j}\right)}\right|^{2}\right]\\
    &=\int_{B_{1}(\zero)}\left(\frac{U\left(x+r'z\right)-U\left(x\right)}{r'}\right)^{2}\frac{\left|\nabla \rho\left(z\right)\right|^{2}}{\rho\left(z\right)}\diff z\\
    &\le \left(\frac{3}{2}\left\|\nabla U\right\|_{\FMC}+\omega_{\nabla U}\left(1\right)\left|x\right|\right)^{2}\int_{B_{1}(\zero)}\frac{\left|\nabla \rho\left(z\right)\right|^{2}}{\rho\left(z\right)}\diff z\\
    &= 2d(d+4)\left(\frac{3}{2}\left\|\nabla U\right\|_{\FMC}+\omega_{\nabla U}\left(1\right)\left|x\right|\right)^{2}\\
    &\le d(d+4)\left(9\left\|\nabla U\right\|_{\FMC}^{2}+4\left(\omega_{\nabla U}(1)\right)^{2}\left|x\right|^{2}\right).
\end{align*}
These properties along with
\begin{align*}
    \E\left[\left|G(x,a_{i\eta})-\nabla \Bar{U}_{r}(x)\right|^{2}\right]=\frac{1}{N_{B}}\E\left[\left|G_{1}(x,a_{i\eta})-\nabla \Bar{U}_{r}(x)\right|^{2}\right]\le \frac{1}{N_{B}}\E\left[\left|G_{1}(x,a_{i\eta})\right|^{2}\right]
\end{align*} 
yield (A5) with $\delta_{\bias,r,0}=\delta_{\bias,r,2}=0$, $\delta_{\variance,0}=9d(d+4)\left\|\nabla U\right\|_{\FMC}^{2}/2N_{B}$, and $\delta_{\variance,2}=2d(d+4)(\omega_{\nabla U}(1))^{2}/N_{B}$ if $r=r'$.
Hence the SG-LMC with this $G$ also can achieve $\Wasserstein_{2}(\mu_{k\eta},\pi)\le \epsilon$ for arbitrary $\epsilon>0$.
Note that the complexity deteriorates by a factor of $\mathcal{O}(d^{3})$ in comparison to SS-LMC; the batch size $N_{B}$ to achieve $\Wasserstein_{2}(\mu_{k\eta},\pi)\le \epsilon$ is of order $\mathcal{O}(d^{5}\ucpi(\log\epsilon^{-1}+d)/\epsilon^{4})$ since $\delta_{\variance,2}=0$ does not hold and both $\delta_{\variance,0}$ and $\delta_{\variance,2}$ are of order $\mathcal{O}(d^{2}/N_{B})$.

We can also consider a zeroth-order version of SS-SG-LMC with the potential $U$ in Eq.~\eqref{eq:emploss} and the following $G$ under (D1)--(D3) and the assumption $|U_{\ell}(x)|<\infty$ for all $\ell=1,\ldots,N$ and $x\in\R^{d}$:
\begin{align*}
    G(x,a_{i\eta})=\frac{1}{N_{B}}\sum_{j=1}^{N_{B}}G_{j}(x,a_{i\eta}):=\frac{1}{N_{B}}\sum_{j=1}^{N_{B}}\frac{U_{\lambda_{i,j}}\left(x+r'\zeta_{i,j}\right)-U_{\lambda_{i,j}}\left(x\right)}{r'}\frac{4\zeta_{i,j}}{(1-|\zeta_{i,j}|^{2})},
\end{align*}
where $N_{B}\in\N$, $r'\in(0,1]$, $a_{i\eta}=[\lambda_{i,1},\ldots,\lambda_{i,N_{B}},\zeta_{i,1},\ldots,\zeta_{i,N_{B}}]$, $\{\lambda_{i,j}\}$ is a sequence of i.i.d.~random variables with the discrete uniform distribution on $\{1,\ldots,N\}$, and $\{\zeta_{i,j}\}$ is a sequence of i.i.d.~random variables with the density $\rho$ and independence of $\{\lambda_{i,j}\}$.
We see that for all $x\in\R^{d}$,
\begin{align*}
    \E\left[G(x,a_{i\eta})\right]=\frac{1}{NN_{B}}\sum_{j=1}^{N_{B}}\sum_{\ell=1}^{N}\int\frac{U_{\ell}\left(x+r'z\right)-U_{\ell}\left(x\right)}{r'}(-\nabla \rho(z))\diff z%\\
    %&=\sum_{k=1}^{N}\int\frac{U\left(x+r'z\right)-U\left(x\right)}{r'}(-\nabla \rho(z))\diff z=
    =\nabla \Bar{U}_{r'}(x)
\end{align*}
and for almost all $x\in\R^{d}$,
\begin{align*}
    \E\left[\left|G(x,a_{i\eta})-\nabla \Bar{U}_{r'}(x)\right|^{2}\right]%&=\frac{1}{N_{B}}\E\left[\left|\frac{U_{\lambda_{i,j}}\left(x+r'\zeta_{i,j}\right)-U_{\lambda_{i,j}}\left(x\right)}{r'}\frac{4\zeta_{i,j}}{(1-|\zeta_{i,j}|^{2})}-\nabla \Bar{U}_{r}(x)\right|^{2}\right]\\
    &\le \frac{1}{N_{B}}\E\left[\left|\frac{U_{\lambda_{i,j}}\left(x+r'\zeta_{i,j}\right)-U_{\lambda_{i,j}}\left(x\right)}{r'}\frac{4\zeta_{i,j}}{(1-|\zeta_{i,j}|^{2})}\right|^{2}\right]\\
    &=\frac{1}{NN_{B}}\sum_{\ell=1}^{N}\int\left|\frac{U_{\ell}\left(x+r'z\right)-U_{\ell}\left(x\right)}{r'}\right|^{2}\frac{\left|\nabla \rho(z)\right|^{2}}{\rho(z)}\diff z\\
    &\le \frac{1}{N_{B}}d(d+4)\max_{\ell=1,\ldots,N}\left(9\left\|\nabla U_{\ell}\right\|_{\FMC}^{2}+4\left(\Hat{\omega}(1)\right)^{2}\left|x\right|^{2}\right).
\end{align*}
Hence (A5) for this $G$ holds with $\delta_{\bias,r',0}=\delta_{\bias,r',2}=0$, $\delta_{\variance,0}=9d(d+4)(\max_{\ell}\left|\nabla U_{\ell}(\zero)\right|+\Hat{\omega}(1))^{2}/2N_{B}$, and $\delta_{\variance,2}=2d(d+4)(\Hat{\omega}(1))^{2}/N_{B}$ if $r=r'$.
Therefore, this SG-LMC can achieve $\Wasserstein_{2}(\mu_{k\eta},\pi)\le \epsilon$ for any $\epsilon>0$, though the complexity is worse than that of SS-SG-LMC by a factor of $\mathcal{O}(d^{2})$.

\section{Preliminary results}\label{sec:preliminary}
We give preliminary results on the compact polynomial mollifier, mollification of functions with the finite moduli of continuity, and the representation of the likelihood ratio between the solutions of SDEs via the Liptser--Shiryaev theory.
We also introduce the fundamental theorem of calculus for weakly differentiable functions, a well-known sufficient condition of Poincar\'{e} inequalities and convergence in $\Wasserstein_{2}$ with the inequalities, and upper bounds for Wasserstein distances.

\subsection{The fundamental theorem of calculus for weakly differentiable functions}

We use the following result on the fundamental theorem of calculus for functions in $W_{\loc}^{1,\infty}(\R^{d})$.

\begin{proposition}[\citealp{lieb2001analysis,anastassiou2009distributional}]\label{prop:fundamental}
    For each $f\in W_{\loc}^{1,\infty}(\R^{d})$, for almost all $x,y\in\R^{d}$,
    \begin{align}\label{eq:FTC}
        f(y)-f(x)=\int_{0}^{1}\left\langle \nabla f\left(x+t\left(y-x\right)\right),y-x\right\rangle \diff t.
    \end{align}
\end{proposition}

\subsection{Properties of the compact polynomial mollifier}\label{sec:mollifier}
We analyse the mollifier $\rho$ proposed in Eq.~\eqref{eq:mollifier}.
Note that our non-asymptotic analysis needs mollifiers of class $\mathcal{C}^{1}$ whose gradients have explicit $L^{1}$-bounds and whose supports are in the unit ball of $\R^{d}$, and it is nontrivial to obtain explicit $L^{1}$-bounds for the gradients of well-known $\mathcal{C}^{\infty}$ mollifiers.

\begin{remark}
    We need mollifiers of class $\mathcal{C}^{1}$ to let $U\ast\rho$ with $U\in W_{\loc}^{1,\infty}(\R^{d})$ be of class $\mathcal{C}^{2}$ and give a bound for the constant of a Poincare inequality by \citet{bakry2008simple}; see Lemma \ref{lem:FMC:C2} and Proposition \ref{prop:BBCG08T14}.
\end{remark}

The following lemma gives some properties of $\rho$.

\begin{lemma}\label{lem:kernel:2}
    (1) $\rho\in\mathcal{C}^{1}(\R^{d})$, (2) $\int \rho(x)\diff x=1$, and (3) $\int |\nabla \rho(x)|\diff x\le d+2$.
\end{lemma}

\begin{proof}
    (1) We check the behaviour of $\nabla \rho$ on a neighbourhood of $\{x\in\R^{d}:|x|=1\}$.
    For all $x\in\R^{d}$ with $|x|<1$,
    \begin{align*}
        \nabla \rho(x)&=\left(\frac{\pi^{d/2}\Beta(d/2,3)}{\Gamma(d/2)}\right)^{-1}\left(-4\right)\left(1-|x|^{2}\right)x
    \end{align*}
    and thus $\nabla \rho(x)$ is continuous at any $x\in\R^{d}$ by $\nabla \rho(x)=\zero$ for all $x\in\R^{d}$ with $|x|=1$.
    
    (2) We have
    \begin{align*}
        \int \rho(x)\diff x&=\frac{2}{\Beta(d/2,3)}\int_{0}^{1} r^{d-1}\left(1-r^{2}\right)^{2}\diff r=\frac{1}{\Beta(d/2,3)}\int_{0}^{1} s^{d/2-1}\left(1-s\right)^{2}\diff s=1
    \end{align*}
    with the change of coordinates from the Euclid one to the hyperspherical one, and the change of variables such that $\sqrt{s}=r$ and $(1/2\sqrt{s})\diff s=\diff r$.

    (3) With respect to the $L^{1}$-norm of the gradient, it holds
    \begin{align*}
        \int |\nabla \rho(x)|\diff x&=\int_{|x|\le 1}\left(\frac{\pi^{d/2}\Beta(d/2,3)}{\Gamma(d/2)}\right)^{-1}4\left(1-|x|^{2}\right)|x|\diff x\\
        &=\frac{8}{\Beta(d/2,3)}\int_{0}^{1}r^{d}\left(1-r^{2}\right)\diff r
        =\frac{4}{\Beta(d/2,3)}\int_{0}^{1}s^{d/2-1/2}\left(1-s\right)\diff s\\
        &=\frac{4\Beta(d/2+1/2,2)}{\Beta(d/2,3)}
        =\frac{4\Gamma(d/2+1/2)\Gamma(2)\Gamma(d/2+3)}{\Gamma(d/2+5/2)\Gamma(d/2)\Gamma(3)}\\
        &=\frac{(d+4)(d+2)d}{(d+3)(d+1)}
        \le d+2
    \end{align*}
    because $(d+4)d\le(d+3)(d+1)$.
    Therefore, the statement holds true.
\end{proof}

We show the optimality of the compact polynomial mollifier; the $L^{1}$-norms of the gradients of $\mathcal{C}^{1}$ non-negative mollifiers with supports in $B_{1}(\zero)$ are bounded below by $d$.

\begin{lemma}
Assume that $p:\R^{d}\to[0,\infty)$ is a continuously differentiable non-negative function whose support is in the unit ball of $\R^{d}$ such that $\int p(x)\diff x=1$.
It holds that
\begin{align*}
    \int_{\R^{d}}\left|\nabla p(x)\right|\diff x\ge d.
\end{align*}
\end{lemma}

\begin{proof}
    Since $p\in \mathcal{C}^{1}(\R^{d})$, the $L^{1}$-norm of the gradient equals the total variation; that is, for arbitrary $R> 1$,
    \begin{align*}
        \int_{\R^{d}}\left|\nabla p(x)\right|\diff x&=\int_{B_{R}\left(\zero\right)}\left|\nabla p(x)\right|\diff x\\
        &=\sup\left\{\int_{B_{R}\left(\zero\right)}p(x)\mathrm{div}\varphi(x)\diff x\left|\varphi\in\mathcal{C}_{0}^{1}\left(B_{R}\left(\zero\right);\R^{d}\right),\left\|\varphi\right\|_{\infty}\le 1\right.\right\},
    \end{align*}
    where $\mathcal{C}_{0}^{1}(B_{R}(\zero);\R^{d})$ is a class of continuously differentiable functions $\varphi:\R^{d}\to\R^{d}$ with compact supports in $B_{R}(\zero)\subset\R^{d}$.
    For all $\delta\in(0,1]$, by fixing $\varphi_{\delta}\in\mathcal{C}_{0}^{1}(B_{R}(\zero);\R^{d})$ such that $\varphi_{\delta}(x)=(1-\delta)x$ for all $x\in B_{1}(\zero)$ and $\|\varphi_{\delta}\|_{\infty}\le 1$, we have
    \begin{align*}
        \int_{\R^{d}}\left|\nabla p(x)\right|\diff x&\ge \int_{B_{R}\left(\zero\right)}p(x)\mathrm{div}\varphi_{\delta}(x)\diff x=\int_{B_{1}\left(\zero\right)}p(x)\mathrm{div}\varphi_{\delta}(x)\diff x=(1-\delta)d.
    \end{align*}
    We obtain the conclusion by taking the limit as $\delta\to0$.
\end{proof}

\subsection{Functions with the finite moduli of continuity and their properties}
We consider a class of possibly discontinuous functions and show lemmas useful for analysis of SG-LMC such that $\nabla U$ and $\widetilde{G}$ are in this class.

Let $\FMC=\FMC(\R^{d};\R^{\ell})$ with fixed $d,\ell\in\N$ denote a class of measurable functions $\phi:(\R^{d},\Borel(\R^{d}))\to(\R^{\ell},\Borel(\R^{\ell}))$ with (1) $|\phi(\zero)|<\infty$ and (2) $\omega_{\phi}(1)<\infty$, where $\omega_{\phi}(\cdot)$ is the well-known modulus of continuity defined as
\begin{align*}
    \omega_{\phi}(r):=\sup_{x,y\in\R^{d}:|x-y|\le r}\left|\phi(x)-\phi(y)\right|,
\end{align*}
where $r>0$.
Note that we use the modulus of continuity not to measure the continuity of $\phi$, but to measure the fluctuation of $\phi$ within $\Bar{B}_{r}(x)$ for all $x\in\R^{d}$.
An intuitive element of $\FMC$ is $\indicator_{A}$ for an arbitrary measurable set $A\in\Borel(\R^{d})$ because $\omega_{\indicator_{A}}(r)\le 1$ for any $A$ and $r>0$.
In the rest of the paper, we sometimes use the notation $\|\phi\|_{\FMC}:=|\phi(\zero)|+\omega_{\phi}(1)$ with $\phi\in\FMC$ just for brevity (it is easy to see that $\FMC$ equipped with $\|\cdot\|_{\FMC}$ is a Banach space).

We introduce the following lemma: this ensures that we can change $r>0$ arbitrarily if $\omega_{\phi}(r)<\infty$ with some $r>0$, and reveal that considering $r=1$ is sufficient to capture the large-scale behaviour since the lemma leads to $\omega_{\phi}(n)\le n\omega_{\phi}(1)$ for any $n\in\N$.
\begin{lemma}\label{lem:FMC:representation}
    For any $r>0$ and $\phi\in\FMC$, $\omega_{\phi}(r)=\sup_{t>0}\lceil t\rceil^{-1}\omega_{\phi}(tr)$.
\end{lemma}

\begin{proof}
    $\omega_{\phi}(r)\le \sup_{t>0}\lceil t\rceil^{-1}\omega_{\phi}(rt)$ and $\omega_{\phi}(r)\ge \lceil t\rceil^{-1}\omega_{\phi}(rt)$ with $t\in(0,1]$ hold immediately.
    Thus we only examine $\omega_{\phi}(r)\ge \lceil t\rceil^{-1}\omega_{\phi}(rt)$ for all $t>1$. 

    We fix $t>1$.
    For any $x,y\in \R^{d}$ with $|x-y|\le rt$,
    \begin{align*}
        \left|\phi\left(x\right)-\phi\left(y\right)\right|&\le \sum_{i=1}^{\lceil t\rceil}\left|\phi\left(\frac{(\lceil t\rceil-i+1)x+(i-1)y}{\lceil t\rceil}\right)-\phi\left(\frac{(\lceil t\rceil-i)x+iy}{\lceil t\rceil}\right)\right|\\
        &\le \lceil t\rceil \omega_{\phi}(r)
    \end{align*}
    because $|((\lceil t\rceil-i+1)x+(i-1)y)/\lceil t\rceil-((\lceil t\rceil-i)x+iy)/\lceil t\rceil|=|x-y|/\lceil t\rceil\le r$.
\end{proof}

\begin{remark}
    Note that the continuity and boundedness of the modulus of continuity do not imply each other.
    For example, $f(x)=x\sin x$ with $x\in\R$ is a continuous function without the finite modulus of continuity.
    On the other hand, $f(x)=\indicator_{\Q}\left(x\right)$ with $x\in\R$ is a trivial example of a function with the finite modulus of continuity and without continuity.

    Moreover, continuity along with the boundedness of the modulus of continuity does not imply uniform continuity, which we can easily observe by $f(x)=\sin(x^{2})$ with $x\in\R$.
\end{remark}

\begin{itembox}[l]{The chains of implications}
\centering
\begin{tikzcd}
        (\textrm{bounded }f) \arrow[dr,Rightarrow] 
        & (\textrm{Lipschitz }f\ast\rho_{r}) 
        & \\
        (\textrm{uniformly continuous }f) \arrow[d,Rightarrow] \arrow[r,Rightarrow] 
        & (\textrm{bounded } \omega_{f}(r)) \arrow[r,Rightarrow] \arrow[u,Rightarrow] \arrow[d,Rightarrow] 
        & (f \textrm{ of linear growth}) \\
        (\textrm{continuous } f)  
        & (\textrm{local Lipschitz }\int f) 
        & 
    \end{tikzcd}
\end{itembox}

\begin{lemma}[linear growth of functions with the finite moduli of continuity]\label{lem:FMC:linearGrowth}
    For any $\phi\in\FMC(\R^{d};\R^{\ell})$, it holds that for all $x\in\R^{d}$,
    \begin{align*}
        \left|\phi\left(x\right)\right| \le \left|\phi(\zero)\right|+\omega_{\phi}(1)+\omega_{\phi}(1)|x|.
    \end{align*}
\end{lemma}

\begin{proof}%[Proof of Lemma \ref{lem:FMC:linearGrowth}]
    Fix $x\in\R^{d}$.
    Lemma \ref{lem:FMC:representation} gives
    \begin{align*}
        \left|\phi\left(x\right)\right|-\left|\phi(\zero)\right|\le \left|\phi\left(x\right)-\phi(\zero)\right|\le \omega_{\phi}(|x|)\le \lceil |x|\rceil\omega_{\phi}(1)\le (1+|x|)\omega_{\phi}(1).
    \end{align*}
    Therefore, the statement holds.
\end{proof}

\begin{lemma}[local Lipschitz continuity by gradients with the finite moduli of continuity]\label{lem:FMC:localLipschitz}
    Assume that $\Phi\in W_{\loc}^{1,\infty}(\R^{d})$ and a representative weak gradient $\nabla \Phi$ is in $\FMC(\R^{d};\R^{d})$.
    It holds that for almost all $x,y\in\R^{d}$,
    \begin{align*}
        \left|\Phi\left(x\right)-\Phi\left(y\right)\right|\le \left(\left|\nabla\Phi(\zero)\right|+\omega_{\nabla\Phi}(1)\left(1+\frac{\left|x\right|+\left|y\right|}{2}\right)\right)\left|y-x\right|.
    \end{align*}
\end{lemma}

\begin{proof}%[Proof of Lemma \ref{lem:FMC:localLipschitz}]
    Proposition \ref{prop:fundamental} and Lemma \ref{lem:FMC:linearGrowth} yield that for almost all $x,y\in\R^{d}$,
    \begin{align*}
        |\Phi(x)-\Phi(y)|&= \left|\int_{0}^{1}\left\langle\nabla \Phi\left(x+t(y-x)\right),y-x\right\rangle\diff t\right|\\
        &\le \int_{0}^{1}\left|\nabla \Phi\left(x+t(y-x)\right)\right|\diff t\left|y-x\right|\\
        &\le \int_{0}^{1}\left(\left|\nabla \Phi\left(\zero\right)\right|+\omega_{\nabla \Phi}\left(1\right)\left(1+\left|(1-t)x+ty\right|\right)\right)\diff t\left|y-x\right|\\
        &\le \left(\left|\nabla \Phi\left(\zero\right)\right|+\omega_{\nabla\Phi}(1)\left(1+\frac{|x|+|y|}{2}\right)\right)\left|y-x\right|.
    \end{align*}
    Hence we obtain the conclusion.
\end{proof}

\begin{lemma}[quadratic growth by gradients with the finite moduli of continuity]\label{lem:FMC:quadGrowth}
    Assume that $\Phi\in W_{\loc}^{1,\infty}(\R^{d})$ and a representative weak gradient $\nabla \Phi$ is in $\FMC(\R^{d};\R^{d})$.
    It holds that $\|\Phi\|_{L^{\infty}(B_{1}(\zero))}<\infty$ and for almost all $x\in\R^{d}$,
    \begin{align*}
        \Phi\left(x\right)&\le \frac{\omega_{\nabla \Phi}(1)}{2}\left|x\right|^{2}+\left(\left|\nabla \Phi\left(\zero\right)\right|+\frac{3}{2}\omega_{\nabla \Phi}(1)\right)\left|x\right|+\left\|\Phi\right\|_{L^{\infty}\left(B_{1}\left(\zero\right)\right)}.
    \end{align*}
    Moreover, for all $x\in\R^{d}$ and $r\in(0,1]$,
    \begin{align*}
        \Bar{\Phi}_{r}(x)&\le \frac{\omega_{\nabla \Phi}(1)}{2}\left|x\right|^{2}+\left(\left|\nabla \Phi\left(\zero\right)\right|+2\omega_{\nabla \Phi}(1)\right)\left|x\right|+\left\|\Phi\right\|_{L^{\infty}\left(B_{1}\left(\zero\right)\right)}
    \end{align*}
    with $\Bar{\Phi}_{r}(x)=(\Phi\ast\rho_{r})(x)$.
\end{lemma}

\begin{proof}
    Lemma \ref{lem:FMC:localLipschitz} gives that for almost all $x\in\R^{d}$ and $y \in B_{1}(\zero)\cap B_{|x|}(x)$,
    \begin{align*}
        \Phi\left(x\right)&\le \frac{\omega_{\nabla \Phi}(1)}{2}\left|x\right|\left|x-y\right|+\left(\left|\nabla \Phi\left(\zero\right)\right|+\frac{3}{2}\omega_{\nabla \Phi}(1)\right)\left|x-y\right|+\Phi\left(y\right)\\
        &\le \frac{\omega_{\nabla \Phi}(1)}{2}\left|x\right|^{2}+\left(\left|\nabla \Phi\left(\zero\right)\right|+\frac{3}{2}\omega_{\nabla \Phi}(1)\right)\left|x\right|+\left\|\Phi\right\|_{L^{\infty}\left(B_{1}\left(\zero\right)\right)}.
    \end{align*}
    Regarding the second statement, it holds that 
    \begin{align*}
        \Bar{\Phi}_{r}(x)&=\int_{\R^{d}} \Phi\left(x-y\right)\rho_{r}\left(y\right)\diff y\\
        &=\int_{\R^{d}} \left(\Phi\left(-y\right)+\int_{0}^{1}\left\langle\nabla \Phi\left(-y+tx\right),x\right\rangle \diff t\right)\rho_{r}\left(y\right)\diff y\\
        &\le \int_{\R^{d}} \left(\Phi\left(-y\right)+\int_{0}^{1}\left(\left|\nabla \Phi(\zero)\right|+\omega_{\nabla \Phi}\left(1\right)\left(1+\left|-y+tx\right|\right)\right)\left|x\right|\diff t\right)\rho_{r}\left(y\right)\diff y\\
        &\le \left\|\Phi\right\|_{L^{\infty}(B_{1}(\zero))}+\int_{0}^{1}\left(\left|\nabla \Phi(\zero)\right|+\omega_{\nabla \Phi}\left(1\right)\left(2+t\left|x\right|\right)\right)\left|x\right|\diff t\\
        &\le \frac{\omega_{\nabla \Phi}(1)}{2}\left|x\right|^{2}+\left(\left|\nabla \Phi\left(\zero\right)\right|+2\omega_{\nabla \Phi}(1)\right)\left|x\right|+\left\|\Phi\right\|_{L^{\infty}(B_{1}(\zero))}.
    \end{align*}
    We obtain the conclusion.
\end{proof}

\begin{lemma}[smoothness of convolution]\label{lem:FMC:C2}
    Assume that $\Phi\in W_{\loc}^{1,\infty}(\R^{d})$ and a representative weak gradient $\nabla \Phi$ is in $\FMC(\R^{d};\R^{d})$.
    Then $\Bar{\Phi}_{r}:=(\Phi\ast\rho_{r})\in\mathcal{C}^{2}(\R^{d})$ and $\nabla^{2}\Bar{\Phi}_{r}=(\nabla\Phi\ast\nabla\rho_{r})$. 
\end{lemma}

\begin{proof}
    Since $\Phi$ and $\nabla \Phi$ are essentially bounded on any compact sets, 
    for some $\{\varphi_{n}\}\subset\mathcal{C}_{0}^{\infty}(\R^{d})$ approximating $\rho_{r}\in C_{0}^{1}(\R^{d})\cap W^{1,1}(\R^{d})$, $\nabla (\Phi\ast\rho_{r})=\Phi\ast\nabla \rho_{r}=\lim_{n}\Phi\ast\nabla \varphi_{n}=\lim_{n}\nabla\Phi\ast\varphi_{n}=\nabla\Phi\ast\rho_{r}$ and thus $\Phi\in\mathcal{C}^{2}(\R^{d})$ with $\nabla^{2}\Bar{\Phi}_{r}=(\nabla\Phi\ast\nabla\rho_{r})$. 
\end{proof}

\begin{lemma}[bounded gradients of convolution]\label{lem:FMC:boundedConvolutedGradient}
For all $\phi\in\FMC(\R^{d};\R^{\ell})$, $r>0$, and $x\in\R^{d}$,
it holds that 
\begin{align*}
    \left\|\nabla \Bar{\phi}_{r}(x)\right\|_{2}\le \left(d+2\right)\frac{\omega_{\phi}(r)}{r},
\end{align*}
where $\Bar{\phi}_{r}(x)=(\phi\ast\rho_{r})(x)$.
\end{lemma}

\begin{proof}%[Proof of Lemma \ref{lem:FMC:boundedConvolutedGradient}]
We obtain
\begin{align*}
    \nabla \left(\Bar{\phi}_{r}\right)(x)
    &=\int_{\R^{d}}\phi\left(y\right)\left(\nabla \rho_{r}\right)\left(x-y\right)\diff y
    =\int_{\R^{d}}\left(\phi\left(y\right)-\phi\left(x\right)\right)\left(\nabla \rho_{r}\right)\left(x-y\right)\diff y
\end{align*}
by using $\int\nabla \rho_{r}(x)\diff x=0$,
and thus
\begin{align*}
    \left\|\nabla \left(\Bar{\phi}_{r}\right)(x)\right\|_{2}&\le \int_{\R^{d}}\left\|\left(\phi\left(y\right)-\phi\left(x\right)\right)\left(\nabla \rho_{r}\right)\left(x-y\right)\right\|_{2}\diff y\\
    &= \int_{\R^{d}}\left|\phi\left(y\right)-\phi\left(x\right)\right|\left|\nabla \rho_{r}\left(x-y\right)\right|\diff y\le \omega_{\phi}(r) \int_{\R^{d}}\left|\nabla \rho_{r}\left(y\right)\right|\diff y\\
    &= \omega_{\phi}(r)\int_{\R^{d}}\left| \frac{1}{r^{d+1}}\nabla\rho\left(\frac{y}{r}\right)\right|\diff y
    = \omega_{\phi}(r)\int_{\R^{d}}\left|\frac{1}{r^{d+1}}\nabla \rho\left(z\right)\right|r^{d}\diff z\\
    &\le \frac{(d+2)\omega_{\phi}(r)}{r} 
\end{align*}
by the change of variables $z=y/r$ with $r^{d}\diff z=\diff y$ and Lemma \ref{lem:kernel:2}.
\end{proof}

% $\ell^{\infty}$ from van der Vaart and Wellner (1996)
\begin{lemma}[$1$-Lipschitz mapping to $\ell^{\infty}$]\label{lem:FMC:lipMap}
    For all $\phi\in\FMC(\R^{d};\R^{\ell})$ and $r>0$,
    \begin{align*}
        \left\|\phi\ast \rho_{r}-\phi\right\|_{\infty}\le \omega_{\phi}(r).
    \end{align*}
    % where $\|f\|_{\infty}:=\sup_{x}|f(x)|$ for all $f:\R^{d}\to\R^{\ell}$.
\end{lemma}

\begin{proof}%[Proof of Lemma \ref{lem:FMC:lipMap}]
    Since $\int\rho_{r}(x)\diff x=1$, for all $x\in\R^{d}$, 
    \begin{align*}
        \left|\phi\ast\rho_{r}(x)-\phi(x)\right|&=\left|\int_{\R^{d}}\phi(y)\rho_{r}(x-y)\diff y-\phi(x)\right|\\
        &=\left|\int_{\R^{d}}\phi(y)\rho_{r}(x-y)\diff y-\int_{\R^{d}}\phi(x)\rho_{r}(x-y)\diff y\right|\\
        &=\left|\int_{\R^{d}}\left(\phi(y)-\phi(x)\right)\rho_{r}(x-y)\diff y\right|\\
        &\le \int_{\R^{d}}\left|\phi(y)-\phi(x)\right|\rho_{r}(x-y)\diff y\\
        &\le \omega_{\phi}(r).
    \end{align*}
    This is the desired conclusion.
\end{proof}

\begin{lemma}[essential supremum of deviations by convolution]\label{lem:FMC:esssupOsc}
    Assume that $\Phi\in W_{\loc}^{1,\infty}(\R^{d})$ and a representative weak gradient $\nabla \Phi$ is in $\FMC(\R^{d};\R^{d})$.
    For all $r>0$,
    \begin{align*}
        \left\|\Bar{\Phi}_{r}-\Phi\right\|_{L^{\infty}(\R^{d})}\le r\omega_{\nabla \Phi}(r)
    \end{align*}
    with $\Bar{\Phi}_{r}(x):=(\Phi\ast\rho_{r})(x)$.
\end{lemma}

\begin{proof}
    By Proposition \ref{prop:fundamental} and $\int_{\R^{d}}\langle y,z\rangle \rho_{r}(y)\diff y=0$ for any $z\in\R^{d}$, for almost all $x\in\R^{d}$,
    \begin{align*}
        \left|\Bar{\Phi}_{r}(x)-\Phi(x)\right|&=\left|\int_{\R^{d}}\left(\Phi(x-y)-\Phi(x)\right)\rho_{r}(y)\diff y\right|\\
        &=\left|\int_{\R^{d}}\left(\int_{0}^{1}\left\langle \nabla \Phi(x-ty),y\right\rangle\diff t\right)\rho_{r}(y)\diff y\right|\\
        &=\left|\int_{\R^{d}}\left(\int_{0}^{1}\left\langle \nabla \Phi(x-ty)-\nabla \Phi(x),y\right\rangle\diff t\right)\rho_{r}(y)\diff y\right|\\
        &\le \omega_{\nabla \Phi}(r)\int_{\R^{d}}|y|\rho_{r}(y)\diff y\\
        &\le r\omega_{\nabla \Phi}(r)
    \end{align*}
    and thus the statement holds.
\end{proof}

\subsection{Liptser--Shiryaev theory}

We show the existence of explicit likelihood ratios of diffusion-type processes based on Theorem 7.19 and Lemma 7.6 of \citet{liptser2001statistics}.
We fix $T>0$ throughout this section.
Let $(W_{T},\mathcal{W}_{T})$ be a measurable space of $\R^{d}$-valued continuous functions $w_{t}$ with $t\in[0,T]$ and $\mathcal{W}_{T}=\sigma(w_{s}:w\in W_{T},s\le T)$.
We also use the notation $\mathcal{W}_{t}=\sigma(w_{s}:w\in W_{T}, s\le t)$ for $t\in[0,T]$.
Let $(\Omega, \mathcal{F}, \mu)$ be a complete probability space and $(\Tilde{\Omega},\Tilde{\mathcal{F}},\Tilde{\mu})$ be its identical copy.
We assume that the filtration $\{\mathcal{F}_{t}\}_{t\in[0,T]}$ satisfies the usual conditions.
Let $(B_{t},\mathcal{F}_{t})$ with $t\in[0,T]$ be a $d$-dimensional Brownian motion and $\xi$ be an $\mathcal{F}_{0}$-measurable $d$-dimensional random vector such that $|\xi|<\infty $ $\mu$-almost surely.
We set $\{a_{t}\}_{t\in[0,T]}$, an $\mathcal{F}_{t}$-adapted random process such that its trajectory $\{a_{s}(\omega)\}_{s\in[0,t]}$ with $\omega\in\Omega$ for each $t\in[0,T]$ is in a measurable space $(A_{t},\mathcal{A}_{t})$.
Assume that $a=\{a_{t}\}_{t\in[0,T]}$, $B=\{B_{t}\}_{t\in[0,T]}$, and $\xi$ are independent of each other.
$\mu_{a},\mu_{B}$, and $\mu_{\xi}$ denote the probability measures induced by $a,B,\xi$ on $(A_{T},\mathcal{A}_{T})$, $(W_{T},\mathcal{W}_{T})$, and $(\R^{d},\Borel(\R^{d}))$ respectively.

Consider the solutions $X^{P}=\{X_{t}^{P}\}_{t\in[0,T]}$ and $X^{Q}=\{X_{t}^{Q}\}_{t\in[0,T]}$ of the following SDEs:
\begin{align}
    \diff X_{t}^{P}&=b^{P}\left(t,a,X^{P}\right)\diff t+\sqrt{2\beta^{-1}}\diff B_{t},\ X_{0}^{P}=\xi,\label{eq:LS:P}\\
    \diff X_{t}^{Q}&=b^{Q}\left(X_{t}^{Q}\right)\diff t+\sqrt{2\beta^{-1}}\diff B_{t},\ X_{0}^{Q}=\xi.\label{eq:LS:Q}
\end{align}

We set the following assumptions, partially adapted from \citet{liptser2001statistics} but containing some differences in $\xi$ and the structure of $X^{Q}$.
\begin{itemize}
    \item[(LS1)] $X_{t}^{P}$ is a strong solution of the equation \eqref{eq:LS:P}, that is, there exists a measurable functional $F_{t}$ for each $t$ such that 
    \begin{align*}
        X_{t}^{P}(\omega)=F_{t}(a(\omega),B(\omega),\xi(\omega))
    \end{align*}
    $\mu$-almost surely.
    \item[(LS2)] $b^{P}$ is non-anticipative, that is, $\mathcal{A}_{t}\times \mathcal{W}_{t}$-measurable for each $t\in[0,T]$, and for fixed $a\in A_{T}$ and $w\in W_{T}$,
    \begin{align*}
        \int_{0}^{T}\left|b^{P}(t,a,w)\right|\diff t<\infty.
    \end{align*}
    \item[(LS3)] $b^{Q}:\R^{d}\to\R^{d}$ is Lipschitz continuous, so that $X^{Q}$ is the unique strong solution of the equation \eqref{eq:LS:Q}.
    \item[(LS4)] It holds that
    \begin{align*}
        &\mu\left(\int_{0}^{T}\left(\left|b^{P}\left(t,a,X^{P}\right)\right|^{2}+\left|b^{Q}\left(X_{t}^{P}\right)\right|^{2}\right)\diff t<\infty\right)\\
        &=\mu\left(\int_{0}^{T}\left(\left|b^{P}\left(t,a,X^{Q}\right)\right|^{2}+\left|b^{Q}\left(X_{t}^{Q}\right)\right|^{2}\right)\diff t<\infty\right)=1.
    \end{align*}
\end{itemize}

We consider a variant of \eqref{eq:LS:P} with fixed $a\in A_{T}$:
\begin{align*}
    \diff X_{t}^{P|a}=b^{P}\left(t,a,X^{P|a}\right)\diff t+\sqrt{2\beta^{-1}}\diff B_{t},\ X_{0}^{P|a}=\xi.
\end{align*}
Then Assumption (LS1) yields that 
\begin{align}\label{eq:LS:cond}
    X_{t}^{P|a}(\omega)=F_{t}(a,B(\omega),\xi(\omega))
\end{align}
$\mu_{a}\times \mu$-almost surely.
We assume that $\Omega=A_{T}\times W_{T}\times \R^{d}$, $\mathcal{F}=\mathcal{A}_{T}\times \mathcal{W}_{T}\times\Borel(\R^{d})$, and $\mu=\mu_{a}\times\mu_{B}\times\mu_{\xi}$ without loss of generality.
Then each $\omega\in\Omega$ has the form $\omega=(a,B,\xi)$ and we can assume that $a,B$, and $\xi$ are projections such as $a(\omega)=a$, $B(\omega)=B$, and $\xi(\omega)=\xi$; therefore, the equality \eqref{eq:LS:cond} holds $\mu_{a}\times\mu_{B}\times\mu_{\xi}$-almost surely.

We consider a process on the product space $(\Omega\times\Tilde{\Omega},\mathcal{F}\times\Tilde{\mathcal{F}},\mu\times \Tilde{\mu})$:
\begin{align*}
    \diff X_{t}^{P|a(\omega)}(\Tilde{\omega})=b^{P}\left(t,a(\omega),X^{P|a(\omega)}(\Tilde{\omega})\right)\diff t+\sqrt{2\beta^{-1}}\diff B_{t}(\Tilde{\omega}),\ X_{0}^{P|a(\omega)}=\xi(\Tilde{\omega}).
\end{align*}
(LS1) gives that 
\begin{align*}
    X_{t}^{P|a(\omega)}(\Tilde{\omega})=F_{t}\left(a(\omega),B(\Tilde{\omega}),\xi(\Tilde{\omega})\right),
\end{align*}
$\mu\times\Tilde{\mu}$-almost surely.

\begin{lemma}\label{lem:LS:L0705}
    Under (LS1), for any $C\in\mathcal{W}_{T}$,
    \begin{align*}
        \mu\left(X^{P}(a,B,\xi)\in C|\sigma(a)\right)=\Tilde{\mu}\left(X^{P}\left(a,\Tilde{B},\Tilde{\xi}\right)\left(\Tilde{\omega}\right)\in C\right)
    \end{align*}
    $\mu$-almost surely.
\end{lemma}

\begin{proof}
    The proof is essentially identical to that of Lemma 7.5 of \citet{liptser2001statistics} except for the randomness of $\xi$.
    We first show that for fixed $t\in[0,T]$ and $C_{t}\in\Borel(\R^{d})$,
    \begin{align*}
        \mu\left(F_{t}(a,B,\xi)\in C_{t}|\sigma(a)\right)=\Tilde{\mu}\left(F_{t}\left(a,\Tilde{B},\Tilde{\xi}\right)\in C_{t}\right)
    \end{align*}
    $\mu$-almost surely.
    Note that the following probability for fixed $a$ is $\mathcal{A}_{T}$-measurable owing to (LS1) and Fubini's theorem:
    \begin{align*}
        \Tilde{\mu}\left(F_{t}\left(a,\Tilde{B},\Tilde{\xi}\right)\in C_{t}\right)=\left(\mu_{B}\times\mu_{\xi}\right)\left(F_{t}(a,B,\xi)\in C_{t}\right).
    \end{align*}
    Let $f(a(\omega))$ be a $\sigma(a)$-measurable bounded random variable.
    Again Fubini's theorem gives that
    \begin{align*}
        \E\left[f(a(\omega))\indicator_{F_{t}(a,B,\xi)\in C_{t}}\right]&=\int_{A_{T}}\int_{W_{T}}\int_{\R^{d}}f(a)\indicator_{F_{t}(a,w,x)\in C_{t}}\mu_{a}(\diff a)\mu_{B}(\diff w)\mu_{\xi}(\diff x)\\
        &=\int_{A_{T}}f(a)\left(\int_{W_{T}}\int_{\R^{d}}\indicator_{F_{t}(a,w,x)\in C_{t}}\mu_{B}(\diff w)\mu_{\xi}(\diff x)\right)\mu_{a}(\diff a)\\
        &=\int_{A_{T}}f(a)\left(\mu_{B}\times \mu_{\xi}\right)\left(F_{t}(a,B,\xi)\in C_{t}\right)\mu_{a}(\diff a)\\
        &=\int_{A_{T}}f(a)\Tilde{\mu}\left(F_{t}\left(a,\Tilde{B},\Tilde{\xi}\right)\in C_{t}\right)\mu_{a}(\diff a)\\
        &=\E\left[f(a)\Tilde{\mu}\left(F_{t}\left(a,\Tilde{B},\Tilde{\xi}\right)\in C_{t}\right)\right]
    \end{align*}
    and thus the definition of conditional expectation yields the result.
    Similarly, we obtain that for all $n\in\N$, $0\le t_{1}<\cdots<t_{n}\le T$, and $C_{t_{i}}\in \Borel(\R^{d})$, $i=1,\ldots,n$,
    \begin{align*}
        &\mu\left(F_{t_{1}}(a,B,\xi)\in C_{t_{1}}\cdots F_{t_{n}}(a,B,\xi)\in C_{t_{n}}|\sigma(a)\right)\\
        &=\Tilde{\mu}\left(F_{t_{1}}\left(a,\Tilde{B},\Tilde{\xi}\right)\in C_{t_{1}},\cdots,F_{t_{n}}\left(a,\Tilde{B},\Tilde{\xi}\right)\in C_{t_{n}}\right).
    \end{align*}
    Therefore, the statement holds true.
\end{proof}

Let $P_{T}$ and $Q_{T}$ denote the laws of $\{(a_{t},X_{t}^{P}):t\in[0,T]\}$ and $\{(a_{t},X_{t}^{Q}):t\in[0,T]\}$.
Note that $a_{t}$ and $X_{t}^{Q}$ are independent of each other by the assumptions.
The following proposition gives the equivalence and the representation of the likelihood ratio.
\begin{proposition}\label{prop:LS01L0706}
    Under (LS1)--(LS4), it holds that
    \begin{align*}
        &\frac{\diff Q_{T}}{\diff P_{T}}\left(a,X^{P}\right)\\
        &=\exp\left(-\sqrt{\frac{\beta}{2}}\int_{0}^{T}\left\langle \left(b^{P}-b^{Q}\right)\left(t,a,X^{P}\right),\diff B_{t}\right\rangle -\frac{\beta}{4}\int_{0}^{T}\left|\left(b^{P}-b^{Q}\right)\left(t,a,X^{P}\right)\right|^{2}\diff t\right).
    \end{align*}
\end{proposition}

\begin{proof}
It is quite parallel to the proof of Lemma 7.6 of \citet{liptser2001statistics}.
For arbitrary set $\Gamma=\Gamma_{1}\times \Gamma_{2}$, $\Gamma_{1}\in\mathcal{A}_{T}$ and $\Gamma_{2}\in\mathcal{W}_{T}$, by Lemma \ref{lem:LS:L0705},
\begin{align*}
    \mu\left(\left(a,X^{P}\right)\in\Gamma\right)&=\int_{A_{T}\times W_{T}\times \R^{d}}\indicator_{a\in \Gamma_{1}}\indicator_{X^{P}(a,w,x)\in\Gamma_{2}}\mu_{a}(\diff a)\mu_{B}\left(\diff w\right)\mu_{\xi}\left(\diff x\right)\\
    &=\int_{a\in \Gamma_{1}}\mu\left(X^{P}(a,B,\xi)\in \Gamma_{2}|\sigma(a)\right)\mu_{a}(\diff a)\\
    &=\int_{a\in \Gamma_{1}}\Tilde{\mu}\left(X^{P}\left(a,\Tilde{B},\Tilde{\xi}\right)\in \Gamma_{2}\right)\mu_{a}(\diff a)\\
    &=\int_{a\in \Gamma_{1}}\left(P|a\right)_{T}(\Gamma_{2})\mu_{a}(\diff a),
\end{align*}
where $(P|a)_{T}$ is the law of \eqref{eq:LS:cond}.
Let $(Q|a)_{T}$ denote the law of $X^{Q}$.
For $\mu_{a}$-almost all $a$, under (LS1)--(LS4) and Theorem 7.19 of \citet{liptser2001statistics}, $(P|a)_{T}\sim (Q|a)_{T}$ and the likelihood ratio is given as
\begin{align*}
    &\frac{\diff (P|a)_{T}}{\diff (Q|a)_{T}}\left(X^{Q}\right)\\
    &=\exp\left(\frac{\beta}{2}\int_{0}^{T}\left\langle \left(b^{P}-b^{Q}\right)\left(t,a,X^{Q}\right),\diff B_{t}\right\rangle -\frac{\beta^{2}}{4}\int_{0}^{T}\left|\left(b^{P}-b^{Q}\right)\left(t,a,X^{Q}\right)\right|^{2}\diff t\right).
\end{align*}
Therefore, we have
\begin{align*}
    \mu\left(\left(a,X^{P}\right)\in\Gamma\right)&=\int_{\Gamma_{1}}\int_{\Gamma_{2}}\left(\frac{\diff (P|a)_{T}}{\diff (Q|a)_{T}}(\diff w)(Q|a)_{T}(\diff w)\right)\mu_{a}(\diff a)\\
    &=\int_{\Gamma_{1}}\int_{\Gamma_{2}}\frac{\diff (P|a)_{T}}{\diff (Q|a)_{T}}(\diff w)\left(\mu_{a}\times (Q|a)_{T}\right)(\diff a \diff w)\\
    &=\int_{\Gamma}\frac{\diff (P|a)_{T}}{\diff (Q|a)_{T}}(\diff w)Q_{T}(\diff a \diff w).
\end{align*}
Since $Q_{T}(a,w:(\diff (P|a)_{T})/(\diff (Q|a)_{T})(w)=0)=0$, Lemma 6.8 of \citet{liptser2001statistics} yields the desired conclusion.
\end{proof}

We obtain the following result.
\begin{proposition}[Kullback--Leibler divergence]\label{prop:KL}
    Under (LS1)--(LS4) and the assumption
    \begin{equation*}
        \E\left[\int_{0}^{T}\left|\left(b^{P}-b^{Q}\right)\left(s,a,X^{P}\right)\right|^{2}\diff s\right]<\infty,
    \end{equation*}
    we obtain
    \begin{align*}
        D\left(P_{T}\left\|Q_{T}\right.\right)= \frac{\beta}{4}\E\left[\int_{0}^{T}\left|\left(b^{P}-b^{Q}\right)\left(s,a,X^{P}\right)\right|^{2}\diff s\right].
    \end{align*}
\end{proposition}

\begin{proof}
    Using Proposition \ref{prop:LS01L0706}, we obtain
    \begin{align*}
        D\left(P_{T}\left\|Q_{T}\right.\right)&=\E\left[\log\left(\frac{\diff P_{T}}{\diff Q_{T}}\right)(a,X^{P})\right]\\
        &=\E\left[\frac{\beta}{4}\int_{0}^{T}\left|\left(b^{P}-b^{Q}\right)\left(s,a,X^{P}\right)\right|^{2}\diff s+\sqrt{\frac{\beta}{2}}\int_{0}^{T}\left\langle \left(b^{P}-b^{Q}\right)\left(s,a,X^{P}\right),\diff B_{s}\right\rangle\right]\\
        &=\frac{\beta}{4}\E\left[\int_{0}^{T}\left|\left(b^{P}-b^{Q}\right)\left(s,a,X^{P}\right)\right|^{2}\diff s\right],
    \end{align*}
    since the local martingale term is a martingale by the assumption.
    Hence we obtain the conclusion.
\end{proof}

\subsection{Poincar\'{e} inequalities}

Let us consider Poincar\'{e} inequalities for a probability measure $P_{\Phi}$ whose density is $(\int e^{-\Phi(x)}\diff x)^{-1}e^{-\Phi(x)}$ with lower bounded $\Phi\in\mathcal{C}^{2}(\R^{d})$ such that $\int e^{-\Phi(x)}\diff x<\infty$.
Let $L:=\Delta -\langle \nabla \Phi,\nabla \rangle$, which is $P_{\Phi}$-symmetric, $P_{t}$ be the Markov semigroup with the infinitesimal generator $L$, and $\mathcal{E}$ denote the Dirichlet form
\begin{align*}
    \mathcal{E}(g):=\lim_{t\to0}\frac{1}{t}\int_{\R^{d}}g \left(g-P_{t}g\right) \diff P_{\Phi},
\end{align*}
where $g\in L^{2}(P_{\Phi})$ such that the limit exists.
Here, we say that a probability measure $P_{\Phi}$ satisfies a Poincar\'{e} inequality with constant $c_{\rm P}(P_{\Phi})$ (the Poincar\'{e} constant) if for any $Q\ll P_{\Phi}$,
\begin{align*}
    \chi^{2}\left(Q\|P_{\Phi}\right)\le c_{\rm P}(P_{\Phi})\mathcal{E}\left(\sqrt{\frac{\diff Q}{\diff P_{\Phi}}}\right).
\end{align*}

We adopt the following statement from \citet{raginsky2017non}; although it is different to the original discussion of \citet{bakry2008simple}, the difference is negligible because Eq.~(2.3) of \citet{bakry2008simple} yields the same upper bound.
\begin{proposition}[\citealp{bakry2008simple}]\label{prop:BBCG08T14}
    Assume that there exists a Lyapunov function $V\in\mathcal{C}^{2}(\R^{d})$ with $V:\R^{d}\to[1,\infty)$ such that
    \begin{align*}
        \frac{LV\left(x\right)}{V\left(x\right)}\le -\lambda_{0}+\kappa_{0}\indicator_{B_{\Tilde{R}}(\zero)}\left(x\right)
    \end{align*}
    for some $\lambda_{0}>0$, $\kappa_{0}\ge 0$ and $\Tilde{R}>0$, where $LV(x)=\Delta V-\langle \nabla \Phi,\nabla V\rangle$.
    Then $P_{\Phi}$ satisfies a Poincar\'{e} inequality with constant $c_{\rm P}(P_{\Phi})$ such that
    \begin{align*}
        c_{\rm P}(P_{\Phi})\le \frac{1}{\lambda_{0}}\left(1+a\kappa_{0}\Tilde{R}^{2}e^{\mathrm{Osc}_{\Tilde{R}}}\right),
    \end{align*}
    where $a>0$ is an absolute constant and $\mathrm{Osc}_{\Tilde{R}}:=\max_{x:|x|\le \Tilde{R}}\Phi(x)-\min_{x:|x|\le \Tilde{R}}\Phi(x)$.
\end{proposition}

The next proposition shows the convergence in $\Wasserstein_{2}$ by $\chi^{2}$-divergence using the recent study by \citet{liu2020poincare}.
\begin{proposition}[\citealp{lehec2023langevin}, Lemma 9]\label{prop:Lehec21L09}
    Assume that $P_{\Phi}$ satisfies Poincar\'{e} inequalities with constant $c_{\rm P}(P_{\Phi})$ and $\nabla \Phi$ is at most of linear growth.
    Then for any probability measure $\nu$ on $(\R^{d},\Borel(\R^{d}))$ with $\nu\ll P_{\Phi}$ and every $t>0$, it holds that
    \begin{align*}
        \Wasserstein_{2}\left(\nu P_{t},P_{\Phi}\right)\le \sqrt{2c_{\rm P}(P_{\Phi})\chi^{2}\left(\nu\|P_{\Phi}\right)}\exp\left(-\frac{t}{2c_{\rm P}(P_{\Phi})}\right),
    \end{align*}
    where $\nu P_{t}$ is the law of the unique weak solution $Z_{t}$ of the SDE
    \begin{align*}
        \diff Z_{t}=-\nabla \Phi\left(Z_{t}\right)\diff t+\sqrt{2}\diff B_{t},\ Z_{0}\sim \nu.
    \end{align*}
\end{proposition}

\subsection{A bound for the 2-Wasserstein distance by KL divergence}

The next proposition is an immediate result by \citet{bolley2005weighted}.
\begin{proposition}[\citealp{bolley2005weighted}]\label{prop:BV05C23}
    Let $\mu,\nu$ be probability measures on $(\R^{d},\Borel(\R^{d}))$.
    Assume that there exists a constant $\lambda>0$ such that $\int \exp(\lambda |x|^{2})\nu(\diff x)<\infty$.
    Then for any $\mu$, it holds that
    \begin{align*}
        \Wasserstein_{2}\left(\mu,\nu\right)\le C_{\nu}\left(D(\mu\|\nu)^{\frac{1}{2}}+\left(\frac{D(\mu\|\nu)}{2}\right)^{\frac{1}{4}}\right),
    \end{align*}
    where
    \begin{align*}
        C_{\nu}:=2\inf_{\lambda>0}\left(\frac{1}{\lambda}\left(\frac{3}{2}+\log\int_{\R^{d}}e^{\lambda\left|x\right|^{2}}\nu\left(\diff x\right)\right)\right)^{\frac{1}{2}}.
    \end{align*}
\end{proposition}

\section{Proof of the main theorem}\label{sec:proof}
In this section, we use the notations $\|\nabla U\|_{\FMC}:=|\nabla U(\zero)|+\omega_{\nabla U}(1)$ and $\|\widetilde{G}\|_{\FMC}:=|\widetilde{G}(\zero)|+\omega_{\widetilde{G}}(1)$ under (A3).
$\Bar{X}_{t}^{r}$ denotes the unique strong solution of the following SDE under (A3) (Lemma \ref{lem:FMC:boundedConvolutedGradient} gives the existence and uniqueness):
\begin{align}\label{eq:sLD}
    \diff \Bar{X}_{t}^{r}=-\nabla \Bar{U}_{r}\left(\Bar{X}_{t}^{r}\right)\diff t+\sqrt{2\beta^{-1}}\diff B_{t},\ \Bar{X}_{0}^{r}=\xi
\end{align}
and $\Bar{\nu}_{t}^{r}$ represents the probability measure of $\Bar{X}_{t}^{r}$.
We use the notations $\pi$ and $\Bar{\pi}^{r}$, probability measures on $(\R^{d},\Borel(\R^{d}))$, as
\begin{align*}
    \pi\left(\diff x\right)=\frac{1}{\mathcal{Z}\left(\beta\right)}\exp\left(-\beta U\left(x\right)\right)\diff x,\ \Bar{\pi}^{r}\left(\diff x\right):=\frac{1}{\Bar{\mathcal{Z}}^{r}\left(\beta\right)}\exp\left(-\beta \Bar{U}_{r}\left(x\right)\right)\diff x,
\end{align*}
where $\mathcal{Z}(\beta)=\int\exp(-\beta U(x))\diff x$ and $\Bar{\mathcal{Z}}^{r}(\beta)=\int\exp(-\beta \Bar{U}_{r}(x))\diff x$.
Note that $\Bar{U}_{r}$ is $(\Bar{m},\Bar{b})$-dissipative with $\Bar{m}:=m,\Bar{b}:=b+\omega_{\nabla U}(1)$ as 
\begin{align*}
    \left\langle x,\nabla \Bar{U}_{r}\left(x\right)\right\rangle&=\int_{\R^{d}}\left\langle x,\nabla U\left(x-y\right)\right\rangle \rho_{r}(y)\diff y\\
    &= \int_{\R^{d}}\left\langle x-y,\nabla U\left(x-y\right)\right\rangle \rho_{r}(y)\diff y\\
    &\quad+\int_{\R^{d}}\left\langle y,\nabla U\left(x-y\right)-\nabla U\left(x\right)\right\rangle \rho_{r}(y)\diff y\\
    &\ge \int_{\R^{d}}\left(m\left| x-y\right|^{2}-b\right) \rho_{r}(y)\diff y-\omega_{\nabla U}(r) \int_{\R^{d}}\left| y\right|\rho_{r}(y)\diff y\\
    &\ge m\left| x\right|^{2}-b+\int_{\R^{d}}\left|y\right|^{2} \rho_{r}(y)\diff y-r\omega_{\nabla U}(r)\\
    &\ge m|x|^{2}-(b+\omega_{\nabla U}(1))
\end{align*}
owing to $r\le 1$ and $\int_{\R^{d}}\langle y,z\rangle \rho_{r}(y)\diff y=0$ for each $z\in\R^{d}$.

\subsection{Moments of SG-LMC algorithms}
\begin{lemma}[uniform $L^{2}$ moments]\label{lem:RRT17L03}
    Assume that (A1)--(A6) hold.
    (1) For all $k\in\N$ and $0<\eta\le 1\wedge (\Tilde{m}/2((\omega_{\widetilde{G}}(1))^{2}+\delta_{\variance,2}))$, $Y_{k\eta},G(Y_{k\eta},a_{k\eta})\in L^{2}$.
    Moreover,
    \begin{align*}
        \sup_{k\ge 0}\E\left[\left|Y_{k\eta}\right|^{2}\right]&\le \kappa_{0}+2\left(1\vee \frac{1}{\Tilde{m}}\right)\left(\Tilde{b}+\left\|\widetilde{G}\right\|_{\FMC}^{2}+\delta_{\variance,0}+\frac{d}{\beta}\right)=:\kappa_{\infty}.
    \end{align*}
    (2) It holds that for any $t\ge0$ and $r\in(0,1]$,
    \begin{align*}
        \E\left[\left|\Bar{X}_{t}^{r}\right|^{2}\right]&\le \kappa_{0}e^{-2\Bar{m}t}+\frac{\Bar{b}+d/\beta}{\Bar{m}}\left(1-e^{-2\Bar{m}t}\right).
    \end{align*}
\end{lemma}

\begin{proof}%[Proof of Lemma \ref{lem:RRT17L03}]
    The proof is adapted from Lemma 3 of \citet{raginsky2017non}.

    (1) We first show $Y_{k\eta}\in L^{2}$ for each $k\in\N$ since
    \begin{align*}
        \E\left[\left.\left|G(Y_{k\eta},a_{k\eta})\right|^{2}\right|Y_{k\eta}\right]&\le  2\E\left[\left.\left|G(Y_{k\eta},a_{k\eta})-\widetilde{G}(Y_{k\eta})\right|^{2}\right|Y_{k\eta}\right]+2\left|\widetilde{G}(Y_{k\eta})\right|^{2}\\
        &\le 4\delta_{\variance,2}\left|Y_{k\eta}\right|^{2}+4\delta_{\variance,0}+\left(4\left\|\widetilde{G}\right\|_{\FMC}+4\left(\omega_{\widetilde{G}}\left(1\right)\right)^{2}\left|Y_{k\eta}\right|^{2}\right)
    \end{align*}
    almost surely and thus $Y_{k\eta}\in L^{2}$ implies $G(Y_{k\eta},a_{k\eta})\in L^{2}$.
    Assumptions (A3), (A5), and (A6) and Lemma \ref{lem:FMC:linearGrowth}  give 
    \begin{align*}
        \E\left[\left|Y_{(k+1)\eta}\right|^{2}\right]
        &=\E\left[\left|Y_{k\eta}-\eta{G}\left(Y_{k\eta},a_{k\eta}\right)+\sqrt{2\beta^{-1}}\left(B_{(k+1)\eta}-B_{k\eta}\right)\right|^{2}\right]\\
        &\le 2\E\left[\left|Y_{k\eta}-\eta{G}\left(Y_{k\eta},a_{k\eta}\right)\right|^{2}\right]+2\E\left[\left|\sqrt{2\beta^{-1}}\left(B_{(k+1)\eta}-B_{k\eta}\right)\right|^{2}\right]\\
        &\le 4\E\left[\left|Y_{k\eta}-\eta \widetilde{G}(Y_{k\eta})\right|^{2}\right]+4\eta^{2}\E\left[\left|\widetilde{G}(Y_{k\eta})-{G}\left(Y_{k\eta},a_{k\eta}\right)\right|^{2}\right]+\frac{4\eta d}{\beta}\\
        % &\le 8\E\left[\left|Y_{k\eta}\right|^{2}+\eta^{2} \left|\widetilde{G}(Y_{k\eta})\right|^{2}\right]+4\eta^{2}\E\left[2\delta_{\variance,2}\left|Y_{k\eta}\right|^{2}+2\delta_{\variance,0}\right]+\frac{4\eta d}{\beta}\\
        &\le \left(8+16(\omega_{\widetilde{G}}(1))^{2}+8\delta_{\variance,2}\right)\E\left[\left|Y_{k\eta}\right|^{2}\right]+\left(16\left\|\widetilde{G}\right\|_{\FMC}^{2}+8\delta_{\variance,0}+\frac{4d}{\beta}\right).
    \end{align*}
    Hence, $Y_{k\eta}\in L^{2}$ as there exist $\gamma_{2},\gamma_{0}>1$ such that $\E[|Y_{(k+1)\eta}|^{2}]\le \gamma_{2}\E[|Y_{k\eta}|^{2}]+\gamma_{0}\le \gamma_{2}^{k+1}\E[|\xi|^{2}]+\gamma_{0}(\gamma_{2}^{k+1}-1)/(\gamma_{2}-1)\le \gamma_{2}^{k+1}(\log\E[\exp(|\xi|^{2})]+\gamma_{0}/(\gamma_{2}-1))\le \gamma_{2}^{k+1}(\kappa_{0}+\gamma_{0}/(\gamma_{2}-1))<\infty$ for arbitrary $k\in\N$ by Jensen's inequality.
    
    The independence among $Y_{k\eta}$, $a_{k\eta}$, and $B_{(k+1)\eta}-B_{k\eta}$ and the square integrability of $Y_{k\eta}$ and $G(Y_{k\eta},a_{k\eta})$ lead to
    \begin{align*}
        \E\left[\left|Y_{(k+1)\eta}\right|^{2}\right]%&=\E\left[\left|Y_{k\eta}-\eta{G}\left(Y_{k\eta},a_{k\eta}\right)+\sqrt{2\beta^{-1}}\left(B_{(k+1)\eta}-B_{k\eta}\right)\right|^{2}\right]\\
        % &=\E\left[\left|Y_{k\eta}-\eta{G}\left(Y_{k\eta},a_{k\eta}\right)\right|^{2}\right]+\frac{2}{\beta}\E\left[\left|B_{(k+1)\eta}-B_{k\eta}\right|^{2}\right]\\
        % &=\E\left[\left|Y_{k\eta}-\eta\widetilde{G}\left(Y_{k\eta}\right)+\eta\widetilde{G}\left(Y_{k\eta}\right)-\eta{G}\left(Y_{k\eta},a_{k\eta}\right)\right|^{2}\right]+\frac{2\eta d}{\beta}\\
        &=\E\left[\left|Y_{k\eta}-\eta \widetilde{G}(Y_{k\eta})\right|^{2}\right]+\eta^{2}\E\left[\left|\widetilde{G}(Y_{k\eta})-{G}\left(Y_{k\eta},a_{k\eta}\right)\right|^{2}\right]+\frac{2\eta d}{\beta}.
    \end{align*}
    Lemma \ref{lem:FMC:linearGrowth} gives
    \begin{align*}
        &\E\left[\left|Y_{k\eta}-\eta \widetilde{G}(Y_{k\eta})\right|^{2}\right]\\
        &=\E\left[\left|Y_{k\eta}\right|^{2}\right]-2\eta\E\left[\left\langle Y_{k\eta}, \widetilde{G}(Y_{k\eta})\right\rangle \right]+\eta^{2}\E\left[\left|\widetilde{G}(Y_{k\eta})\right|^{2}\right]\\
        &\le \E\left[\left|Y_{k\eta}\right|^{2}\right]+2\eta\left(\Tilde{b}-\Tilde{m}\E\left[\left|Y_{k\eta}\right|^{2}\right]\right)+2\eta^{2}\left(\left\|\widetilde{G}\right\|_{\FMC}^{2}+\left(\omega_{\widetilde{G}}(1)\right)^{2}\E\left[\left|Y_{k\eta}\right|^{2}\right]\right)\\
        &=\left(1-2\eta \Tilde{m}+2\eta^{2}\left(\omega_{\widetilde{G}}(1)\right)^{2}\right)\E\left[\left|Y_{k\eta}\right|^{2}\right]+2\eta \Tilde{b}+2\eta^{2}\left\|\widetilde{G}\right\|_{\FMC}^{2}.
    \end{align*}
    By Assumption (A5) and the independence between $a_{k\eta}$ and $Y_{k\eta}$, we also have
    \begin{align*}
        \E\left[\left|\widetilde{G}(Y_{k\eta})-{G}\left(Y_{k\eta},a_{k\eta}\right)\right|^{2}\right]\le 2\delta_{\variance,2}\E\left[\left|Y_{k\eta}\right|^{2}\right]+2\delta_{\variance,0}.
    \end{align*}
    Hence it holds that for $\gamma:=1-2\eta \Tilde{m}+2\eta^{2}((\omega_{\widetilde{G}}(1))^{2}+\delta_{\variance,2})<1$,
    \begin{align*}
        \E\left[\left|Y_{(k+1)\eta}\right|^{2}\right]
        &\le \gamma\E\left[\left|Y_{k\eta}\right|^{2}\right]+2\eta \Tilde{b}+2\eta^{2}\left\|\widetilde{G}\right\|_{\FMC}^{2}+2\eta^{2}\delta_{\variance,0}+\frac{2\eta d}{\beta}.
    \end{align*}
    If $\gamma\le 0$, then it is obvious that
    \begin{align*}
        \E\left[\left|Y_{(k+1)\eta}\right|^{2}\right]\le 2\Tilde{b}+2\left\|\widetilde{G}\right\|_{\FMC}^{2}+2\delta_{\variance,0}+\frac{2d}{\beta},
    \end{align*}
    and if $\gamma\in(0,1)$,
    \begin{align*}
        \E\left[\left|Y_{k\eta}\right|^{2}\right]&\le \gamma^{k}\E\left[\left|Y_{0}\right|^{2}\right]+\frac{2\eta \Tilde{b}+2\eta^{2}\left\|\widetilde{G}\right\|_{\FMC}^{2}+2\eta^{2}\delta_{\variance,0}+\frac{2\eta d}{\beta}}{2\eta \Tilde{m}-2\eta^{2}\left(\left(\omega_{\widetilde{G}}(1)\right)^{2}+\delta_{\variance,2}\right)}\\
        &\le \E\left[\left|Y_{0}\right|^{2}\right]+\frac{2\Tilde{b}+2\eta\left\|\widetilde{G}\right\|_{\FMC}^{2}+2\eta \delta_{\variance,0}+\frac{2d}{\beta}}{2\Tilde{m}-2\eta\left(\left(\omega_{\widetilde{G}}(1)\right)^{2}+\delta_{\variance,2}\right)}\\
        &\le \kappa_{0}+\frac{2}{\Tilde{m}}\left(\Tilde{b}+\left\|\widetilde{G}\right\|_{\FMC}^{2}+\delta_{\variance,0}+\frac{d}{\beta}\right)
    \end{align*}
    since Jensen's inequality yields $\E[|\xi|^{2}]\le \log\E[\exp(|\xi|^{2})]= \kappa_{0}$.
    
    (2) It\^{o}'s formula yields
    \begin{align*}
        e^{2\Bar{m}t}\left|\Bar{X}_{t}^{r}\right|^{2}&=\left|\Bar{X}_{0}^{r}\right|^{2}+\int_{0}^{t}\left(e^{2\Bar{m}s}\left\langle -\nabla U\left(\Bar{X}_{s}^{r}\right),2\Bar{X}_{s}^{r}\right\rangle+e^{2\Bar{m}s}\frac{2d}{\beta}+2\Bar{m}e^{2\Bar{m}s}\left|\Bar{X}_{s}^{r}\right|^{2}\right)\diff s\\
        &\quad+\sqrt{2\beta^{-1}}\int_{0}^{t}e^{2\Bar{m}s}\left\langle 2\Bar{X}_{s}^{r},\diff B_{s}\right\rangle.
    \end{align*}
    The dissipativity and the martingale property of the last term lead to
    \begin{align*}
        \E\left[\left|\Bar{X}_{t}^{r}\right|^{2}\right]&=e^{-2\Bar{m}t}\E\left[\left|\xi\right|^{2}\right]\\
        &\quad+2\int_{0}^{t}e^{2\Bar{m}(s-t)}\left(\E\left[\left\langle -\nabla U\left(\Bar{X}_{s}^{r}\right),\Bar{X}_{s}^{r}\right\rangle+\Bar{m}\left|\Bar{X}_{s}^{r}\right|^{2}\right]+\frac{d}{\beta}\right)\diff s\\
        &\le e^{-2\Bar{m}t}\E\left[\left|
        \xi\right|^{2}\right]+2\int_{0}^{t}e^{2\Bar{m}(s-t)}\left(\E\left[-\Bar{m}\left|\Bar{X}_{s}^{r}\right|^{2}+\Bar{b}+\Bar{m}\left|\Bar{X}_{s}^{r}\right|^{2}\right]+\frac{d}{\beta}\right)\diff s\\
        &\le e^{-2\Bar{m}t}\kappa_{0}+\frac{\Bar{b}+d/\beta}{\Bar{m}}\left(1-e^{-2\Bar{m}t}\right).
    \end{align*}
    We obtain the conclusion.
\end{proof}

\begin{lemma}[exponential integrability of mollified Langevin dynamics]\label{lem:LD:exponential}
    Assume (A1)--(A4) and (A6).
    For all $r\in(0,1]$ and $\alpha\in(0,\beta \Bar{m}/2)$ such that $\E[\exp(\alpha|\xi|^{2})]<\infty$, 
    \begin{align*}
        \E\left[\exp\left(\alpha |\Bar{X}_{t}^{r}|^{2}\right)\right]&\le  \E\left[e^{\alpha |\xi|^{2}}\right]e^{-2\alpha \left(\Bar{b}+d/\beta\right)t}+2 \exp\left(\frac{2\alpha(\Bar{b}+d/\beta)}{\Bar{m}-2\alpha/\beta}\right)\left(1- e^{-2\alpha \left(\Bar{b}+d/\beta\right)t}\right).
    \end{align*}
    In particular, for $\alpha=1\wedge (\beta \Bar{m}/4)$,
    \begin{align*}
        \sup_{t\ge0}\log\E\left[e^{\alpha|\Bar{X}_{t}^{r}|^{2}}\right]&\le \log\left(e^{\alpha\kappa_{0}}\vee 2e^{4\alpha(\Bar{b}+d/\beta)/\Bar{m}}\right)\le \alpha\kappa_{0}+\frac{4\alpha(\Bar{b}+d/\beta)}{\Bar{m}}+1.
    \end{align*}
\end{lemma}

\begin{proof}%[Proof of Lemma \ref{lem:LD:exponential}]
    Let $V_{\alpha}(x):=\exp(\alpha|x|^{2})$.
    Note that
    \begin{align*}
        \nabla V_{\alpha}(x)= 2\alpha V_{\alpha}(x)x,\ 
        \nabla^{2} V_{\alpha}(x)=4\alpha^{2} V_{\alpha}(x)xx^{\top}+2\alpha V_{\alpha}(x)I_{d}.
    \end{align*}
    Let $\Bar{\mathcal{L}}^{r}$ denote the extended generator of $\Bar{X}_{t}^{r}$ such that $\Bar{\mathcal{L}}^{r}f:=\beta^{-1}\Delta f-\langle \nabla \Bar{U}_{r}, \nabla f\rangle$ for $f\in\mathcal{C}^{2}(\R^{d})$.
    It holds that
    \begin{align*}
        \Bar{\mathcal{L}}^{r}V_{\alpha}(x)&\le -2\alpha V_{\alpha}(x)\left\langle \nabla \Bar{U}_{r}(x),x\right\rangle +2(\alpha/\beta)V_{\alpha}(x)\left(2\alpha\left|x\right|^{2}+d\right)\\
        &\le 2\alpha V_{\alpha}(x)\left(\left(-\Bar{m}\left|x\right|^{2}+\Bar{b}\right) +\left(2\alpha\left|x\right|^{2}+d\right)/\beta\right)\\
        &=2\alpha V_{\alpha}(x)\left(\left(2\alpha/\beta-\Bar{m}\right)\left|x\right|^{2}+\Bar{b}+ d/\beta\right).
    \end{align*}
    Let $R^{2}= 2(\Bar{b}+d/\beta)/(\Bar{m}-2\alpha/\beta)$ be a fixed constant and then we obtain for all $x\in\R^{d}$ with $|x|\ge R$,
    \begin{align*}
        \Bar{\mathcal{L}}^{r}V_{\alpha}(x)&\le -2\alpha \left(\Bar{b}+ d/\beta\right)V_{\alpha}(x)
    \end{align*}
    and trivially for all $x\in\R^{d}$ with $|x|<R$,
    \begin{align*}
        \Bar{\mathcal{L}}^{r}V_{\alpha}(x)&\le 2\alpha e^{2\alpha(\Bar{b}+d/\beta)/(\Bar{m}-2\alpha/\beta)}\left(\Bar{b}+d/\beta\right)\\
        &\le 4\alpha e^{2\alpha(\Bar{b}+d/\beta)/(\Bar{m}-2\alpha/\beta)}\left(\Bar{b}+d/\beta\right)-2\alpha \left(\Bar{b}+ d/\beta\right)V_{\alpha}(x).
    \end{align*}
    Thus we have for all $x\in\R^{d}$,
    \begin{align*}
        \Bar{\mathcal{L}}^{r}V_{\alpha}(x)\le -2\alpha \left(\Bar{b}+ d/\beta\right)V_{\alpha}(x)+4\alpha e^{2\alpha(\Bar{b}+d/\beta)/(\Bar{m}-2\alpha/\beta)}\left(\Bar{b}+d/\beta\right).
    \end{align*}
    By It\^{o}'s formula, there exists a sequence of stopping times $\{\sigma_{n}\in[0,\infty)\}_{n\in\N}$ with $\sigma_{n}<\sigma_{n+1}$ for all $n\in\N$ and $\sigma_{n}\uparrow \infty$ as $n\to\infty$ almost surely such that for all $n\in\N$ and $t\ge0$,
    \begin{align*}
        &\E\left[e^{2\alpha \left(\Bar{b}+ d/\beta\right)(t\wedge\sigma_{n})}V_{\alpha}\left(\Bar{X}_{t\wedge\sigma_{n}}^{r}\right)\right]\\
        &= \E\left[V_{\alpha}\left(\Bar{X}_{0}^{r}\right)\right]\\
        &\quad+\E\left[\int_{0}^{t\wedge\sigma_{n}}\left(e^{2\alpha \left(\Bar{b}+ d/\beta\right)s}\Bar{\mathcal{L}}^{r}V_{\alpha}\left(\Bar{X}_{s}^{r}\right)+2\alpha \left(\Bar{b}+ d/\beta\right)e^{2\alpha \left(\Bar{b}+ d/\beta\right)s}V_{\alpha}\left(\Bar{X}_{s}^{r}\right)\right)\diff s\right].
    \end{align*}
    It holds that
    \begin{align*}
        &\E\left[\int_{0}^{t\wedge\sigma_{n}}\left(e^{2\alpha \left(\Bar{b}+ d/\beta\right)s}\Bar{\mathcal{L}}^{r}V_{\alpha}\left(\Bar{X}_{s}^{r}\right)+2\alpha \left(\Bar{b}+ d/\beta\right)e^{2\alpha \left(\Bar{b}+ d/\beta\right)s}V_{\alpha}\left(\Bar{X}_{s}^{r}\right)\right)\diff s\right]\\
        &\le 4\alpha e^{2\alpha(\Bar{b}+d/\beta)/(\Bar{m}-2\alpha/\beta)}\left(\Bar{b}+d/\beta\right)\E\left[\int_{0}^{t\wedge\sigma_{n}}e^{2\alpha \left(\Bar{b}+ d/\beta\right)s}\diff s\right]\\
        &\le 4\alpha e^{2\alpha(\Bar{b}+d/\beta)/(\Bar{m}-2\alpha/\beta)}\left(\Bar{b}+d/\beta\right)\int_{0}^{t}e^{2\alpha \left(\Bar{b}+ d/\beta\right)s}\diff s
    \end{align*}
    and thus Fatou's lemma gives
    \begin{align*}
        &\E\left[e^{2\alpha \left(\Bar{b}+ d/\beta\right)t}V_{\alpha}\left(\Bar{X}_{t}^{r}\right)\right]\\
        &=\E\left[\lim_{n\to\infty} e^{2\alpha \left(\Bar{b}+ d/\beta\right)(t\wedge\sigma_{n})}V_{\alpha}\left(\Bar{X}_{t\wedge\sigma_{n}}^{r}\right)\right]\\
        &\le \liminf_{n\to\infty}\E\left[e^{2\alpha \left(\Bar{b}+ d/\beta\right)(t\wedge\sigma_{n})}V_{\alpha}\left(\Bar{X}_{t\wedge\sigma_{n}}^{r}\right)\right]\\
        &\le \E\left[V_{\alpha}\left(\Bar{X}_{0}^{r}\right)\right]+4\alpha e^{2\alpha(\Bar{b}+d/\beta)/(\Bar{m}-2\alpha/\beta)}\left(\Bar{b}+d/\beta\right)\int_{0}^{t}e^{2\alpha \left(\Bar{b}+ d/\beta\right)s}\diff s.
    \end{align*}
    Therefore,
    \begin{align*}
        \E\left[V_{\alpha}\left(\Bar{X}_{t}^{r}\right)\right]&\le \E\left[e^{\alpha \left|\Bar{X}_{0}^{r}\right|^{2}}\right]e^{-2\alpha \left(\Bar{b}+ d/\beta\right)t}+2 e^{2\alpha(\Bar{b}+d/\beta)/(\Bar{m}-2\alpha/\beta)}\left(1- e^{-2\alpha \left(\Bar{b}+ d/\beta\right)t}\right)
    \end{align*}
    and we obtain the desired conclusion.
\end{proof}

\subsection{Poincar\'{e} inequalities for distributions with mollified potentials}

Let $\Bar{L}^{r}$ be an operator such that $\Bar{L}^{r}f:=\Delta f-\beta\langle \nabla \Bar{U}_{r},\nabla f\rangle $ for all $f\in\mathcal{C}^{2}(\R^{d})$.
Note that Lemma \ref{lem:FMC:C2} yields $\Bar{U}_{r}\in\mathcal{C}^{2}(\R^{d})$.

\begin{lemma}[a bound for the constant of a Poincar\'{e} inequality for $\Bar{\pi}^{r}$]\label{lem:FI:PI:smooth}
    Under (A1)--(A4), for some absolute constant $a>0$, for all $r\in(0,1]$,
    \begin{align*}
        c_{\rm P}(\Bar{\pi}^{r})&\le \frac{2}{\Bar{m}\beta\left(d+\Bar{b}\beta\right)}+\frac{4a\left(d+\Bar{b}\beta\right)}{\Bar{m}\beta}\exp\left(\beta\left(\frac{3}{2}\left\|\nabla U\right\|_{\FMC}\left(1+\frac{4\left(d+\Bar{b}\beta\right)}{\Bar{m}\beta}\right)+U_{0}\right)\right),
    \end{align*}
    where $U_{0}:=\left\|U\right\|_{L^{\infty}(B_{1}(\zero))}<\infty$.
\end{lemma}
\begin{remark}
    Note that this upper bound is independent of $r$.
\end{remark}

\begin{proof}
    We adapt the discussion of \citet{raginsky2017non}.
    We set a Lyapunov function $V(x)=e^{\Bar{m}\beta|x|^{2}/4}$.
    Since $-\langle \nabla \Bar{U}_{r}(x),x\rangle \le -\Bar{m}|x|^{2}+\Bar{b}$ for all $x\in\R^{d}$, it holds that
    \begin{align*}
        \Bar{L}^{r}V(x)&=\left(\frac{d\Bar{m}\beta}{2}+\frac{\left(\Bar{m}\beta\right)^{2}}{4}\left|x\right|^{2}-\frac{\Bar{m}\beta^{2}}{2}\left\langle \nabla \Bar{U}_{r}\left(x\right),x\right\rangle \right)V(x)\\
        &\le \left(\frac{d\Bar{m}\beta}{2}+\frac{\left(\Bar{m}\beta\right)^{2}}{4}\left|x\right|^{2}-\frac{\Bar{m}^{2}\beta^{2}}{2}\left|x\right|^{2}+\frac{\Bar{m}\beta^{2}\Bar{b}}{2}\right)V(x)\\
        &= \left(\frac{\Bar{m}\beta\left(d+\Bar{b}\beta\right)}{2}-\frac{\Bar{m}^{2}\beta^{2}}{4}|x|^{2}\right)V(x).
    \end{align*}
    We fix the constants
    \begin{align*}
        \kappa=\frac{\Bar{m}\beta\left(d+\Bar{b}\beta\right)}{2},\ \gamma=\frac{\Bar{m}^{2}\beta^{2}}{4},\ \Tilde{R}^{2}=\frac{2\kappa}{\gamma}=\frac{4\left(d+\Bar{b}\beta\right)}{\Bar{m}\beta}.
    \end{align*}
    Lemma \ref{lem:FMC:quadGrowth}, $2a\le a^{2}+1$ for $a>0$, and $U(x)\ge 0$ give $\left\|U\right\|_{L^{\infty}(B_{1}(\zero))}<\infty$ and
    \begin{align*}
        \text{Osc}_{\Tilde{R}}\left(\beta \Bar{U}_{r}\right)&\le \beta\left(\frac{\omega_{\nabla U}(1)}{2}\Tilde{R}^{2}+2\left\|\nabla U\right\|_{\FMC}\Tilde{R}+\left\|U\right\|_{L^{\infty}(B_{1}(\zero))}\right)\\
        &\le \beta\left(\left(\left\|\nabla U\right\|_{\FMC}+\frac{\omega_{\nabla U}(1)}{2}\right)\Tilde{R}^{2}+\left\|\nabla U\right\|_{\FMC}+\left\|U\right\|_{L^{\infty}(B_{1}(\zero))}\right)\\
        &\le \beta\left(\frac{3}{2}\left\|\nabla U\right\|_{\FMC}\Tilde{R}^{2}+\left\|\nabla U\right\|_{\FMC}+\left\|U\right\|_{L^{\infty}(B_{1}(\zero))}\right)\\
        &\le \beta\left(\frac{3}{2}\left\|\nabla U\right\|_{\FMC}\left(1+\Tilde{R}^{2}\right)+\left\|U\right\|_{L^{\infty}(B_{1}(\zero))}\right).
    \end{align*}
    Proposition \ref{prop:BBCG08T14} with $\lambda_{0}=\kappa_{0}=\kappa$ yields that for some absolute constant $a>0$,
    \begin{align*}
        c_{\rm P}\left(\Bar{\pi}^{r}\right)&\le \frac{2}{\Bar{m}\beta\left(d+\Bar{b}\beta\right)}\left(1+a\frac{\Bar{m}\beta\left(d+\Bar{b}\beta\right)}{2}\frac{4\left(d+\Bar{b}\beta\right)}{\Bar{m}\beta}e^{\text{Osc}_{\Tilde{R}}\left(\beta \Bar{U}_{r}\right)}\right)\\
        &=\frac{2}{\Bar{m}\beta\left(d+\Bar{b}\beta\right)}+\frac{4a\left(d+\Bar{b}\beta\right)}{\Bar{m}\beta}\exp\left(\beta\left(\frac{3}{2}\left\|\nabla U\right\|_{\FMC}\left(1+\frac{4\left(d+\Bar{b}\beta\right)}{\Bar{m}\beta}\right)+U_{0}\right)\right).
    \end{align*}
    Hence the statement holds true.
\end{proof}

\subsection{Kullback--Leibler and $\chi^{2}$-divergences}

\begin{lemma}\label{lem:RRT17L07}
    Under (A1)--(A6), for any $k\in\N$ and $\eta\in(0,1\wedge (\Tilde{m}/2((\omega_{\widetilde{G}}(1))^{2}+\delta_{\variance,2}))]$, it holds true that
    \begin{align*}
        D(\mu_{k\eta}\|\Bar{\nu}_{k\eta}^{r})\le \left(C_{0}\frac{\omega_{\nabla U}(r)}{r}\eta+\beta\left(\delta_{r,2}\kappa_{\infty}+\delta_{r,0}\right)\right)k\eta,
    \end{align*}
    where $C_{0}$ is a positive constant such that
    \begin{align*}
        C_{0}&=\left(d+2\right)\left(\frac{\beta}{3}\left(\left\|\widetilde{G}\right\|_{\FMC}^{2}+\delta_{\variance,0}+\left((\omega_{\widetilde{G}}(1))^{2}+\delta_{\variance,2}\right)\kappa_{\infty}\right)+\frac{d}{2}\right).
    \end{align*}
\end{lemma}

\begin{proof}%[Proof of Lemma \ref{lem:RRT17L07}]
    We set $A_{t}:=\{\mathbf{a}_{s}=a_{\lfloor s/\eta\rfloor \eta}:a_{i\eta}\in A,i=0,\ldots,\lfloor t/\eta\rfloor,s\le t\}$ with $t\le k\eta$ and $\mathcal{A}_{t}:=\sigma(\{\mathbf{a}\in A_{t}:\mathbf{a}_{s_{j}}\in C_{j},j=1,\ldots,n\}:s_{j}\in[0,t],C_{j}\in\mathcal{A},n\in\N)$.
    Let $P_{k\eta}$ and $Q_{k\eta}$ denote the probability measures on $(A_{k\eta}\times W_{k\eta},\mathcal{A}_{k\eta}\times \mathcal{W}_{k\eta})$ of $\{(a_{\lfloor t/\eta\rfloor \eta},Y_{t}):0\le t\le T\}$ and $\{(a_{\lfloor t/\eta\rfloor \eta},\Bar{X}_{t}^{r}):0\le t\le T\}$ respectively.
    Note that $\Bar{X}_{t}^{r}$ is the unique strong solution to Eq.~\eqref{eq:sLD}
    and such a unique strong solution of this equation exists for any $r>0$ since $\nabla \Bar{U}_{r}$ is Lipschitz continuous by Lemma \ref{lem:FMC:boundedConvolutedGradient}.
    We obtain
    \begin{align*}
        &\frac{\beta}{4}\E\left[\int_{0}^{k\eta}\left|\nabla \Bar{U}_{r}\left(Y_{t}\right)-G\left(Y_{\lfloor t/\eta\rfloor \eta},a_{\lfloor t/\eta\rfloor\eta}\right)\right|^{2}\diff t\right]\\
        &\le \frac{\beta}{2}\sum_{j=0}^{k-1}\E\left[\int_{j\eta}^{(j+1)\eta}\left|\nabla \Bar{U}_{r}\left(Y_{t}\right)-\nabla \Bar{U}_{r}\left(Y_{\lfloor t/\eta\rfloor \eta}\right)\right|^{2}\diff t\right]\\
        &\quad+\frac{\beta}{2}\sum_{j=0}^{k-1}\E\left[\int_{j\eta}^{(j+1)\eta}\left|\nabla \Bar{U}_{r}\left(Y_{\lfloor t/\eta\rfloor \eta}\right)-G\left(Y_{\lfloor t/\eta\rfloor \eta},a_{\lfloor t/\eta\rfloor\eta}\right)\right|^{2}\diff t\right]\\
        &\le \frac{\beta}{2}\frac{\left(d+2\right)\omega_{\nabla U}(r)}{r}\sum_{j=0}^{k-1}\E\left[\int_{j\eta}^{(j+1)\eta}\left|Y_{t}-Y_{\lfloor t/\eta\rfloor \eta}\right|^{2}\diff t\right]\\
        &\quad+\frac{\beta}{2}\sum_{j=0}^{k-1}\E\left[\eta\left|\nabla \Bar{U}_{r}\left(Y_{j \eta}\right)-G\left(Y_{j\eta},a_{j\eta}\right)\right|^{2}\right].
    \end{align*}
    Note that $\E[\langle G(Y_{j\eta},a_{j\eta})-\widetilde{G}(Y_{j\eta}), f(Y_{j\eta})\rangle ]=0$ for any measurable $f:\R^{d}\to\R^{d}$ of linear growth since $Y_{j\eta}$ is square integrable by Lemma \ref{lem:RRT17L03} and $\sigma(Y_{(j-1)\eta},a_{(j-1)\eta},B_{j\eta}-B_{(j-1)\eta})$-measurable, and $a_{j\eta}$ is independent of this $\sigma$-algebra. 
    For all $t\in[j\eta,(j+1)\eta)$, by Lemmas \ref{lem:FMC:linearGrowth} and \ref{lem:RRT17L03},
    \begin{align*}
        &\E\left[\left|Y_{t}-Y_{\lfloor t/\eta\rfloor \eta}\right|^{2}\right]\\
        &=\E\left[\left|-\left(t-j\eta\right)G\left(Y_{j\eta},a_{j\eta}\right)+\sqrt{2\beta^{-1}}\left(B_{t}-B_{j\eta}\right)\right|^{2}\right]\\
        %&=\left(t-j\eta\right)^{2}\E\left[\left|G\left(Y_{j\eta},a_{j\eta}\right)\right|^{2}\right]+2\beta^{-1}\E\left[\left|B_{t}-B_{j\eta}\right|^{2}\right]\\
        &=\left(t-j\eta\right)^{2}\E\left[\left|G\left(Y_{j\eta},a_{j\eta}\right)-\widetilde{G}\left(Y_{j\eta}\right)+\widetilde{G}\left(Y_{j\eta}\right)\right|^{2}\right]+2\beta^{-1}\E\left[\left|B_{t}-B_{j\eta}\right|^{2}\right]\\
        &\le \left(t-j\eta\right)^{2}\left(2\delta_{\variance,2}\E\left[\left|Y_{j\eta}\right|^{2}\right]+2\delta_{\variance,0}+\E\left[\left|\widetilde{G}\left(Y_{j\eta}\right)\right|^{2}\right]\right)+2\beta^{-1}d\left(t-j\eta\right)\\
        &\le 2\left(t-j\eta\right)^{2}\left(\left\|\widetilde{G}\right\|_{\FMC}^{2}+\delta_{\variance,0}+\left((\omega_{\widetilde{G}}(1))^{2}+\delta_{\variance,2}\right)\E\left[\left|Y_{j\eta}\right|^{2}\right]\right)+2\beta^{-1}d\left(t-j\eta\right)\\
        &\le 2\left(t-j\eta\right)^{2}\left(\left\|\widetilde{G}\right\|_{\FMC}^{2}+\delta_{\variance,0}+\left((\omega_{\widetilde{G}}(1))^{2}+\delta_{\variance,2}\right)\kappa_{\infty}\right)+2\beta^{-1}d\left(t-j\eta\right)\\
        &=:2\left(t-j\eta\right)^{2}C'+2\beta^{-1}d\left(t-j\eta\right)
    \end{align*}
    and thus
    \begin{align*}
        \sum_{j=0}^{k-1}\E\left[\int_{j\eta}^{(j+1)\eta}\left|Y_{t}-Y_{\lfloor t/\eta\rfloor \eta}\right|^{2}\diff t\right]\le \left(\frac{2C'\eta^{3}}{3}+\frac{d\eta^{2}}{\beta}\right)k\le \left(\frac{2C'}{3}+\frac{d}{\beta}\right)k\eta^{2}.
    \end{align*}
    It holds that
    %Since $Y_{j\eta}$ is $\sigma(Y_{(j-1)\eta},a_{(j-1)\eta},B_{j\eta}-B_{(j-1)\eta})$-measurable and $a_{j\eta}$ is independent of this $\sigma$-algebra,
    \begin{align*}
        \E\left[\left|\nabla \Bar{U}_{r}\left(Y_{j\eta}\right)-G\left(Y_{j\eta},a_{j\eta}\right)\right|^{2}\right]
        %&=\E\left[\E\left[\left.\left|\nabla \Bar{U}_{r}\left(Y_{j\eta}\right)-G\left(Y_{j\eta},a_{j\eta}\right)\right|^{2}\right|Y_{(j-1) \eta},a_{(j-1)\eta},B_{j\eta}-B_{(j-1)\eta}\right]\right]\\
        &\le \E\left[2\left(\delta_{\bias,r,2}+\delta_{\variance,2}\right)\left|Y_{j\eta}\right|^{2}+2\left(\delta_{\bias,r,0}+\delta_{\variance,0}\right)\right]\\
        &\le 2\delta_{r,2}\kappa_{\infty}+2\delta_{r,0}.
    \end{align*}
    Assumptions (LS1)--(LS4) of Propositions \ref{prop:LS01L0706} and \ref{prop:KL} are satisfied owing to (A1)--(A6), Lemma \ref{lem:RRT17L03}, and the linear growths of $\widetilde{G}(w_{\lfloor \cdot /\eta\rfloor \eta})$ with respect to $\max_{i=0,\ldots,k}|w_{i\eta}|$ and $\nabla \Bar{U}_{r}(w_{t})$ with respect to $|w_{t}|$.
    Therefore, the data-processing inequality and Proposition \ref{prop:KL} give
    \begin{align*}
        D(\mu_{k\eta}\|\Bar{\nu}_{k\eta}^{r})
        &\le \int\log\left(\frac{\diff P_{k\eta}}{\diff  Q_{k\eta}}\right)\diff P_{k\eta}\\
        &= \frac{\beta}{4}\E\left[\int_{0}^{k\eta}\left|\nabla \Bar{U}_{r}\left(Y_{t}\right)-G\left(Y_{\lfloor t/\eta\rfloor \eta},a_{\lfloor t/\eta\rfloor\eta}\right)\right|^{2}\diff t\right]\\
        &\le \frac{(d+2)\omega_{\nabla U}(r)}{r}\left(\frac{C'\beta}{3}+\frac{d}{2}\right)k\eta^{2}+\beta\left(\delta_{r,2}\kappa_{\infty}+\delta_{r,0}\right)k\eta\\
        &=\left(\frac{(d+2)\omega_{\nabla U}(r)}{r}\left(\frac{C'\beta}{3}+\frac{d}{2}\right)\eta+\beta\left(\delta_{r,2}\kappa_{\infty}+\delta_{r,0}\right)\right)k\eta\\
        &=\left(C_{0}\frac{\omega_{\nabla U}(r)}{r}\eta+\beta\left(\delta_{r,2}\kappa_{\infty}+\delta_{r,0}\right)\right)k\eta.
    \end{align*}
    This is the desired conclusion.
\end{proof}

\begin{lemma}[Lemma 2 of \citealp{raginsky2017non}]\label{lem:RRT17L02}
    Under (A1) and (A4), for almost all $x\in\R^{d}$,
    \begin{align*}
        U(x)\ge \frac{m}{3}|x|^{2}-\frac{b}{2}\log3.
    \end{align*}
\end{lemma}

\begin{proof}%[Proof of Lemma \ref{lem:RRT17L02}]
    The proof is adapted from Lemma 2 of \citet{raginsky2017non}.
    We first fix $c\in(0,1]$.
    Since $\{x\in\R^{d}:x\text{ or }cx\text{ is in the set such that Eq.~\eqref{eq:FTC} does not hold}\}$ is null, for almost all $x\in\R^{d}$,
    \begin{align*}
        U(x)&=U(cx)+\int_{0}^{1}\left\langle \nabla U(cx+t(x-cx)),x-cx\right\rangle \diff t\\\
        &\ge \int_{0}^{1}\left\langle \nabla U((c+t(1-c))x),(1-c)x\right\rangle \diff t\\
        &=\int_{0}^{1}\frac{1-c}{c+t(1-c)}\left\langle \nabla U((c+t(1-c))x),(c+t(1-c))x\right\rangle \diff t\\
        &\ge \int_{0}^{1}\frac{1-c}{c+t(1-c)}\left(m(c+t(1-c))^{2}|x|^{2}-b\right) \diff t\\
        &=\int_{c}^{1}\frac{1}{s}\left(ms^{2}|x|^{2}-b\right) \diff s\\
        &=\frac{1-c^{2}}{2}m|x|^{2}+b\log c.
    \end{align*}
    Here, $s=c+t(1-c)$ and thus $\diff t=(1-c)^{-1}\diff s$.
    $c=1/\sqrt{3}$ yields the conclusion.
\end{proof}
\begin{lemma}\label{lem:chi2}
    Under (A1)--(A4) and (A6), it holds that for all $r\in(0,1]$,
    \begin{align*}
        \chi^{2}(\mu_{0}\|\Bar{\pi}^{r})\le 3^{\beta b/2}\left(\frac{3\psi_{2}}{m\beta}\right)^{d/2}\exp\left(\beta\left(2\left\|\nabla U\right\|_{\FMC}+U_{0}\right)+2\psi_{0}\right).
    \end{align*}
\end{lemma}

\begin{proof}
    The density of $\Bar{\pi}^{r}$ is given as $(\diff\Bar{\pi}^{r}/\diff x)(x)=\Bar{\mathcal{Z}}^{r}(\beta)^{-1}e^{-\beta \Bar{U}_{r}\left(x\right)}$ and 
    \begin{align*}
        \chi^{2}(\mu_{0}\|\Bar{\pi}^{r})&=\int_{\R^{d}}\left(\left(\frac{\left(\int_{\R^{d}}e^{-\Psi(x)}\diff x\right)^{-1}e^{-\Psi(x)}}{\Bar{\mathcal{Z}}^{r}(\beta)^{-1}e^{-\beta \Bar{U}_{r}\left(x\right)}}\right)^{2}-1\right)\Bar{\mathcal{Z}}^{r}(\beta)^{-1}e^{-\beta \Bar{U}_{r}\left(x\right)}\diff x\\
        &=\frac{\Bar{\mathcal{Z}}^{r}(\beta)}{\int_{\R^{d}}e^{-\Psi(x)}\diff x}\int_{\R^{d}}e^{\beta \Bar{U}_{r}\left(x\right)-\Psi\left(x\right)}\mu_{0}(\diff x)-1\\
        &\le \frac{e^{\psi_{0}}\Bar{\mathcal{Z}}^{r}(\beta)}{\int_{\R^{d}}e^{-\psi_{2}|x|^{2}}\diff x}\int_{\R^{d}}e^{\beta\left(\frac{\omega_{\nabla U}(1)}{2}\left|x\right|^{2}+2\left\|\nabla U\right\|_{\FMC}\left|x\right|+\left\|U\right\|_{L^{\infty}\left(B_{1}\left(\zero\right)\right)}\right)-\Psi(x)}\mu_{0}(\diff x)\\
        &\le \frac{e^{\psi_{0}}\Bar{\mathcal{Z}}^{r}(\beta)}{(\pi/\psi_{2})^{d/2}}\int_{\R^{d}}e^{\beta\left(\left\|\nabla U\right\|_{\FMC}\left|x\right|^{2}+2\left\|\nabla U\right\|_{\FMC}+\left\|U\right\|_{L^{\infty}\left(B_{1}\left(\zero\right)\right)}\right)-\Psi(x)}\mu_{0}(\diff x)\\
        &\le \frac{\Bar{\mathcal{Z}}^{r}(\beta)}{(\pi/\psi_{2})^{d/2}}e^{\beta\left(2\left\|\nabla U\right\|_{\FMC}+\left\|U\right\|_{L^{\infty}\left(B_{1}\left(\zero\right)\right)}\right)+2\psi_{0}}
    \end{align*}
    by Lemma \ref{lem:FMC:quadGrowth} and $2|x|\le |x|^{2}/2+2$.
    Lemma \ref{lem:RRT17L02}, Jensen's inequality, and the convexity of $e^{-x}$ yield
    \begin{align*}
        \Bar{\mathcal{Z}}^{r}(\beta)&=\int_{\R^{d}}e^{-\beta \Bar{U}_{r}\left(x\right)}\diff x\\
        &=\int_{\R^{d}}e^{-\beta \int _{\R^{d}}U(x-y)\rho_{r}(y)\diff y}\diff x\\
        &\le \int_{\R^{d}}\int_{\R^{d}}e^{-\beta U(x-y)}\diff x\rho_{r}(y)\diff y\\
        &\le e^{\frac{1}{2}\beta b\log 3}\int_{\R^{d}}\int_{\R^{d}}e^{-m\beta|x-y|^{2}/3}\diff x\rho_{r}(y)\diff y\\
        &=3^{\beta b/2}\left(3\pi/m\beta\right)^{d/2}.
    \end{align*}
    Here we obtain the conclusion.
\end{proof}

\begin{lemma}[Kullback--Leibler divergence of Gibbs distributions]\label{lem:KLGibbs}
    Under (A1)--(A4) and (A6), it holds that 
    \begin{align*}
        D\left(\pi\|\Bar{\pi}^{r}\right)\le \beta r\omega_{\nabla U}(r).
    \end{align*}
\end{lemma}
\begin{proof}%[Proof of Lemma \ref{lem:KLGibbs}]
    
The divergence of $\pi$ from $\Bar{\pi}^{r}$ is
\begin{align*}
    D\left(\pi\|\Bar{\pi}^{r}\right)&=\frac{1}{\mathcal{Z}(\beta)}\int \exp\left(-\beta U\left(x\right)\right)\log\left[\frac{\Bar{\mathcal{Z}}^{r}(\beta)\exp\left(-\beta U\left(x\right)\right)}{\mathcal{Z}(\beta)\exp\left(-\beta \Bar{U}_{r}\left(x\right)\right)}\right]\diff x\\
    &=\frac{\beta}{\mathcal{Z}(\beta)}\int \exp\left(-\beta U\left(x\right)\right)\left(\Bar{U}_{r}(x)-U\left(x\right)\right)\diff x+\left(\log\Bar{\mathcal{Z}}^{r}(\beta)-\log\mathcal{Z}(\beta)\right).
\end{align*}
Lemma \ref{lem:FMC:esssupOsc} yields
\begin{align*}
    \frac{\beta}{\mathcal{Z}(\beta)}\int_{\R^{d}} \left(\Bar{U}_{r}(x)-U\left(x\right)\right)e^{-\beta U\left(x\right)}\diff x
    &\le \frac{\beta}{\mathcal{Z}(\beta)}\int_{\R^{d}} \left|\Bar{U}_{r}(x)-U\left(x\right)\right|e^{-\beta U\left(x\right)}\diff x\\
    &\le \frac{\beta }{\mathcal{Z}(\beta)}\int_{\R^{d}} r\omega_{\nabla U}(r)e^{-\beta U\left(x\right)}\diff x\\
    &\le \beta r\omega_{\nabla U}(r).
\end{align*}
Jensen's inequality and Fubini's theorem give
\begin{align*}
    \Bar{\mathcal{Z}}^{r}\left(\beta\right)&=\int_{\R^{d}}\exp\left(-\beta \int_{\R^{d}}U\left(x-y\right)\rho_{r}\left(y\right)\diff y\right)\diff x\\
    &\le \int_{\R^{d}}\int_{\R^{d}}\exp\left(-\beta U\left(x-y\right)\right)\rho_{r}\left(y\right)\diff y\diff x\\
    &=\int_{\R^{d}}\int_{\R^{d}}\exp\left(-\beta U\left(x-y\right)\right)\diff x\rho_{r}\left(y\right)\diff y\\
    &= \mathcal{Z}\left(\beta\right)
\end{align*}
and thus $\log\Bar{\mathcal{Z}}^{r}(\beta)-\log\mathcal{Z}(\beta)\le 0$.
\end{proof}

\subsection{Proof of Theorem \ref{thm:sglmc}}
We complete the proof of Theorem \ref{thm:sglmc}.

\begin{proof}[Proof of Theorem \ref{thm:sglmc}]
We decompose the 2-Wasserstein distance as follows:
    \begin{align*}
    \Wasserstein_{2}(\mu_{k\eta},\pi)\le \underbrace{\Wasserstein_{2}(\mu_{k\eta},\Bar{\nu}_{k\eta}^{r})}_\textrm{(1)}+\underbrace{\Wasserstein_{2}(\Bar{\nu}_{k\eta}^{r},\Bar{\pi}^{r})}_\textrm{(2)}+\underbrace{\Wasserstein_{2}(\Bar{\pi}^{r},\pi)}_\textrm{(3)}.
\end{align*}

(1) We first consider an upper bound for $\Wasserstein_{2}(\mu_{k\eta},\Bar{\nu}_{k\eta}^{r})$.
Proposition \ref{prop:BV05C23} gives
\begin{align*}
    \Wasserstein_{2}(\mu_{k\eta},\Bar{\nu}_{k\eta}^{r})\le C_{\Bar{\nu}_{k\eta}^{r}}\left(D\left(\mu_{k\eta}\|\Bar{\nu}_{k\eta}^{r}\right)^{\frac{1}{2}}+\left(\frac{D\left(\mu_{k\eta}\|\Bar{\nu}_{k\eta}^{r}\right)}{2}\right)^{\frac{1}{4}}\right),
\end{align*}
where
\begin{align*}
    C_{\Bar{\nu}_{k\eta}^{r}}:=2\inf_{\lambda>0}\left(\frac{1}{\lambda}\left(\frac{3}{2}+\log\int_{\R^{d}}e^{\lambda\left|x\right|^{2}}\Bar{\nu}_{k\eta}^{r}\left(\diff x\right)\right)\right)^{\frac{1}{2}}.
\end{align*}
We fix $\lambda=1\wedge (\beta \Bar{m}/4)$ and then Lemma \ref{lem:LD:exponential} leads to
\begin{align*}
    C_{\Bar{\nu}_{k\eta}^{r}}&\le \frac{1}{\lambda^{1/2}}\left(6+4\log\int_{\R^{d}}e^{\lambda\left|x\right|^{2}}\Bar{\nu}_{k\eta}^{r}\left(\diff x\right)\right)^{\frac{1}{2}}\\
    &\le \frac{1}{\lambda^{1/2}}\left(6+4\left(\lambda\kappa_{0}+\frac{4\lambda(\Bar{b}+d/\beta)}{\Bar{m}}+1\right)\right)^{\frac{1}{2}}\\
    &\le \left(4\kappa_{0}+\frac{16(\Bar{b}+d/\beta)}{\Bar{m}}+\frac{10}{1\wedge (\beta \Bar{m}/4)}\right)^{\frac{1}{2}}.
\end{align*}
Hence Lemma \ref{lem:RRT17L07} gives the following bound:
\begin{align*}
    \Wasserstein_{2}(\mu_{k\eta},\Bar{\nu}_{k\eta}^{r})
    &\le C_{1}\max\left.\left\{x^{\frac{1}{2}},x^{\frac{1}{4}}\right\}\right|_{x=\left(C_{0}\frac{\omega_{\nabla U}(r)}{r}\eta+\beta\left(\delta_{r,2}\kappa_{\infty}+\delta_{r,0}\right)\right)k\eta}.
\end{align*}

(2)
In the second place, let us give a bound for $\Wasserstein_{2}(\Bar{\nu}_{k\eta}^{r},\Bar{\pi}^{r})$.
Proposition \ref{prop:Lehec21L09} and Lemma \ref{lem:chi2} yield
\begin{align*}
    \Wasserstein_{2}(\Bar{\nu}_{k\eta}^{r},\Bar{\pi}^{r})&\le \sqrt{2c_{\rm P}(\Bar{\pi}^{r})\chi^{2}\left(\mu_{0}\|\Bar{\pi}^{r}\right)}\exp\left(-\frac{t}{2\beta c_{\rm P}(\Bar{\pi}^{r})}\right)\\
    &\le \sqrt{2c_{\rm P}(\Bar{\pi}^{r})C_{2}}\exp\left(-\frac{t}{2\beta c_{\rm P}(\Bar{\pi}^{r})}\right).
\end{align*}

(3)
Thirdly, we consider a bound for $\Wasserstein_{2}(\Bar{\pi}^{r},\pi)$.
Proposition \ref{prop:BV05C23} gives
\begin{align*}
    \Wasserstein_{2}(\Bar{\pi}^{r},\pi)\le C_{\Bar{\pi}^{r}}\left(D\left(\pi\|\Bar{\pi}^{r}\right)^{\frac{1}{2}}+\left(\frac{D\left(\pi\|\Bar{\pi}^{r}\right)}{2}\right)^{\frac{1}{4}}\right),
\end{align*}
where $C_{\Bar{\pi}^{r}}:=2\inf_{\lambda>0}\left(\frac{1}{\lambda}\left(\frac{3}{2}+\log\int_{\R^{d}}e^{\lambda\left|x\right|^{2}}\Bar{\pi}^{r}\left(\diff x\right)\right)\right)^{\frac{1}{2}}$.
We fix $\lambda=1\wedge (\beta \Bar{m}/4)$ and then Lemmas \ref{lem:LD:exponential} along with Fatou's lemma leads to
\begin{align*}
    C_{\Bar{\pi}^{r}}
    &\le \left(\frac{16(\Bar{b}+d/\beta)}{\Bar{m}}+\frac{10}{1\wedge (\beta \Bar{m}/4)}\right)^{\frac{1}{2}}.
\end{align*}
Lemma \ref{lem:KLGibbs} yields the bound
\begin{align*}
    \Wasserstein_{2}(\Bar{\pi}^{r},\pi)\le C_{1}\max\left.\left\{y^{\frac{1}{2}},y^{\frac{1}{4}}\right\}\right|_{y=\beta r\omega_{\nabla U}(r)}.
\end{align*}

(4)
By (1) and (3), 
\begin{align*}
    \Wasserstein_{2}(\mu_{k\eta},\Bar{\nu}_{k\eta}^{r})+\Wasserstein_{2}(\Bar{\pi}^{r},\pi)&\le C_{1}\left(\max\left\{x^{\frac{1}{2}},x^{\frac{1}{4}}\right\}+\max\left\{y^{\frac{1}{2}},y^{\frac{1}{4}}\right\}\right)\\
    &\le C_{1}\left(2\left(x+y\right)^{\frac{1}{2}}+2\left(x+y\right)^{\frac{1}{4}}\right),
\end{align*}
where
\begin{align*}
    x=\left(C_{0}\frac{\omega_{\nabla U}(r)}{r}\eta+\beta\left(\delta_{r,2}\kappa_{\infty}+\delta_{r,0}\right)\right)k\eta,\
    {y=\beta r\omega_{\nabla U}(r)}.
\end{align*}
Hence, we obtain the desired conclusion.
\end{proof}

\section*{Acknowledgements}
The author was supported by JSPS KAKENHI Grant Number JP21K20318 and JST CREST Grant Numbers JPMJCR21D2 and JPMJCR2115.

\bibliographystyle{apalike}
\bibliography{bibSGLMC}

\end{document}